\documentclass{article}

\usepackage{amsthm}
\usepackage{amsmath, amssymb}
\usepackage{hyperref}
\usepackage{tikz}
\usetikzlibrary{matrix}
\usepackage[matrix, arrow, curve]{xy}

\newtheorem{Thm}{Theorem}[section]
\newtheorem{Def}[Thm]{Definition}
\newtheorem{Lem}[Thm]{Lemma}
\newtheorem{Prop}[Thm]{Proposition}
\newtheorem{Kor}[Thm]{Corollary}
\newtheorem{Rem}[Thm]{Remark}

\title{The cancellation of projective modules of rank $2$ with a trivial determinant}

\author{Tariq Syed\\
	Fakult\"at f\"ur Mathematik\\
	Universit\"at Duisburg-Essen\\
	Thea-Leymann-Stra{\ss}e 9\\
	D-45127 Essen\\
	tariq.syed@gmx.de}

\date{\today}
\begin{document}

\maketitle

\begin{abstract}
We study the cancellation property of projective modules of rank $2$ with a trivial determinant over Noetherian rings of dimension $\leq 4$. If $R$ is a smooth affine algebra of dimension $4$ over an algebraically closed field $k$ such that $6 \in k^{\times}$, then we prove that stably free $R$-modules of rank $2$ are free if and only if a Hermitian $K$-theory group $\tilde{V}_{SL} (R)$ is trivial.
\end{abstract}

\tableofcontents

\section{Introduction}

Let $R$ be a commutative ring. An important question in the study of projective modules is under which circumstances a finitely generated projective $R$-module $P$ is cancellative, i.e. under which circumstances any isomorphism $P \oplus R^{k} \cong Q \oplus R^{k}$ for some finitely generated projective $R$-module $Q$ and $k > 0$ already implies $P \cong Q$. Therefore one studies the fibers of the stabilization maps

\begin{center}
$\phi_{r}: \mathcal{V}_{r} (R) \rightarrow \mathcal{V}_{r+1} (R), [P] \mapsto [P \oplus R]$,
\end{center}

from the set of isomorphism classes of finitely generated projective $R$-modules of constant rank $r$ to the set of isomorphism classes of finitely generated projective $R$-modules of constant rank $r+1$. For any finitely generated projective $R$-module $P$ of rank $r$, the fiber $\phi_{r}^{-1} ([P \oplus R])$ can be described as the orbit space $Um (P \oplus R)/Aut (P \oplus R)$ of the set $Um (P \oplus R)$ of epimorphisms $P \oplus R \rightarrow R$ under the right action of the group $Aut (P \oplus R)$ of automorphisms of $P \oplus R$ given by pre-composition. Note that if $P$ is free, we can identify the set $Um (P \oplus R)$ with the set $Um_{r+1} (R)$ of unimodular rows of length $r+1$ over $R$ and the group $Aut (P \oplus R)$ with the group $GL_{r+1} (R)$. Similarly, one is interested in the orbit spaces $Um (P \oplus R)/SL (P \oplus R)$ and $Um (P \oplus R)/E (P \oplus R)$, where $SL (P \oplus R)$ is the subgroup of $Aut (P \oplus R)$ of automorphisms of determinant $1$ and $E (P \oplus R)$ is its subgroup generated by elementary automorphisms. 
If $P$ is free, then we can identify $SL (P \oplus R)$ and $E (P \oplus R)$ with the subgroups $SL_{r+1} (R)$ and $E_{r+1} (R)$ of $GL_{r+1} (R)$ respectively.\\
Now let $R$ be a Noetherian commutative ring of dimension $d \geq 3$ and let $P$ be a projective $R$-module of constant rank $r$. It follows from \cite[Chapter IV, Theorem 3.4]{HB} that $Um (P \oplus R)/E (P \oplus R)$ is trivial if $r \geq d+1$; in particular, $P$ is cancellative in this case. Furthermore, it was proven in \cite{S1} that $P$ is cancellative if $r = d$ whenever $R$ is an affine algebra over an algebraically closed field. It follows from \cite{B} that the same result holds for affine algebras over a perfect field $k$ with $c.d. (k) \leq 1$ and $d! \in k^{\times}$. In \cite{FRS}, it was proven that stably free modules of rank $d-1$ are free whenever $R$ is a smooth affine algebra over an algebraically closed field $k$ such that $(d-1)! \in k^{\times}$. Nevertheless, it follows from \cite{NMK} that, for any algebraically closed field $k$, there exists a smooth affine $k$-algebra $R$ of dimension $d=4$ which admits a non-free stably free module of rank $2$. The cancellation of projective modules of rank $2$ over smooth affine algebras of dimension $4$ over algebraically closed fields can hence be considered a particularly subtle problem. Therefore it is of interest to find a general criterion for a projective $R$-module of rank $2$ with a trivial determinant over such algebras to be cancellative.\\
Now recall that we have defined in \cite{Sy} a generalized Vaserstein symbol

\begin{center}
$V_{\theta_{0}}: Um (P_{0} \oplus R)/E (P_{0} \oplus R) \rightarrow \tilde{V} (R)$
\end{center}

associated to any projective module $P_{0}$ of rank $2$ over any commutative ring $R$ with a fixed trivialization $\theta_{0}: R \xrightarrow{\cong} det(P_{0})$ of its determinant. The group $\tilde{V} (R)$ is canonically isomorphic to the so-called elementary symplectic Witt group $W_E (R)$ (cp. \cite[\S 3]{SV}). In this paper, we also prove that the generalized Vaserstein symbol descends to a map

\begin{center}
$V_{\theta_{0}}: Um (P_{0} \oplus R)/SL (P_{0} \oplus R) \rightarrow \tilde{V}_{SL} (R)$,
\end{center}

which we call the generalized Vaserstein symbol modulo $SL$ (cp. Theorem \ref{T3.1}). The group $\tilde{V}_{SL} (R)$ is the cokernel of a hyperbolic map $SK_1 (R) \rightarrow \tilde{V} (R)$.\\
For a Noetherian ring $R$ of dimension $\leq 4$, we then give criteria for the surjectivity and injectivity of the generalized Vaserstein symbol modulo SL (cp. Theorem \ref{T3.2} and Theorem \ref{SLInjective}). As a consequence, we obtain a criterion for the triviality of the orbit space $Um (P_{0} \oplus R)/SL (P_{0} \oplus R)$ (cp. Corollary \ref{C3.7}). By further examining this criterion for smooth affine algebras of dimension $4$ over algebraically closed fields (cp. Corollary \ref{C3.18}), we prove our main result in this paper (cp. Theorem \ref{T3.19}):\\\\
\textbf{Theorem.} Let $R$ be a $4$-dimensional smooth affine algebra over an algebraically closed field $k$ with $6 \in k^{\times}$. Then stably free $R$-modules of rank $2$ are free if and only if $\tilde{V}_{SL} (R) = 0$.\\\\
The group $\tilde{V}_{SL} (R)$ is a Hermitian $K$-theory group and hence forms part of a theory which behaves in many aspects like a cohomology theory; this makes the group $\tilde{V}_{SL} (R)$ as computable as one might hope. In the situation of the theorem, it is actually a $2$-torsion group (cp. the proof of Corollary \ref{C3.20}). In particular, $\tilde{V}_{SL} (R) = 0$ if and only if $\tilde{V}_{SL} (R)$ is $2$-divisible. Motivated by this, we use the Gersten-Grothendieck-Witt spectral sequence in order to find cohomological criteria for the $2$-divisibility of the groups $\tilde{V} (R)$ and $\tilde{V}_{SL} (R)$ (cp. Proposition \ref{P2.8}). These criteria enable us to give cohomological criteria for the triviality of stably free modules of rank $2$ (cp. Corollary \ref{C3.20}). As an immediate consequence, we obtain the following result (cp. Corollary \ref{C3.21} in the text):\\\\
\textbf{Corollary.} Let $R$ be a $4$-dimensional smooth affine algebra over an algebraically closed field $k$ with $6 \in k^{\times}$ and let $X = Spec(R)$. Moreover, assume that $CH^{i} (X) = 0$ for $i=1,2,3,4$ and $H^{2} (X, \textbf{I}^{3}) = 0$. Then all finitely generated projective $R$-modules are componentwise free.\\\\
In this corollary, we denote by $CH^{i} (X)$ for $i=1,2,3,4$ the Chow groups of cycles of codimension $i$ on $X$ and by $\textbf{I}^{3}$ the sheaf associated to the third power of the fundamental ideal in the Witt ring (cp. \cite[Section 3.1]{Mo}).\\
Note that the cohomological criteria in the corollary are satisfied if $X$ is stably $\mathbb{A}^{1}$-contractible. The generalized Serre conjecture on algebraic vector bundles asserts that all algebraic vector bundles over topologically contractible smooth affine complex varieties are trivial; the conjecture is known to hold in dimensions $\leq 2$, but is open in higher dimensions. Another open conjecture asserts that topologically contractible smooth affine complex varieties are stably $\mathbb{A}^{1}$-contractible (cp. \cite[Conjecture 5.3.11]{AO}). By our corollary above, this would imply the generalized Serre conjecture in dimension $4$ (see Remark \ref{R3.22}).\\
The paper is organized as follows: We first recall in Section \ref{2.1} how oriented projective modules which are stably isomorphic to a given oriented finitely generated projective module $(P, \theta)$ can be classified in terms of the orbit space $Um (P \oplus R)/SL (P \oplus R)$. In Sections \ref{2.2} and \ref{2.3} we define the groups $W'_{SL} (R)$ and $V_{SL} (R)$ for any commutative ring $R$ and discuss in particular the isomorphism $W'_{SL} (R) \cong V_{SL} (R)$. Sections \ref{2.4} and \ref{2.5} serve as a brief introduction to motivic homotopy theory and higher Grothendieck-Witt groups. While most of this is purely expository, we will then study the Gersten-Grothendieck-Witt spectral sequence at the end of Section \ref{2.4} in order to prove some new cohomological criteria for the $2$-divisibility of $\tilde{V} (R)$ and $\tilde{V}_{SL} (R)$ whenever $R$ is a smooth affine algebra of dimension $4$ over an algebraically closed field $k$ with $char(k) \neq 2$. In Section \ref{3.1}, we introduce the generalized Vaserstein symbol modulo $SL$ and prove the main results of this paper in the subsequent sections. In particular, we prove our criteria for the surjectivity and injectivity of the generalized Vaserstein symbol modulo $SL$ in Section \ref{3.2} as well as our criterion for the triviality of the orbit space $Um (P_{0} \oplus R)/SL (P_{0} \oplus R)$ whenever $R$ is a Noetherian commutative ring of dimension $\leq 4$. In the subsequent section, we can in fact give descriptions of the orbit spaces $Um (P_{0} \oplus R)/E (P_{0} \oplus R)$ and $Um (P_{0} \oplus R)/SL (P_{0} \oplus R)$. Finally, Section \ref{3.4} is dedicated to the study of symplectic orbits of unimodular rows.

\subsection*{Acknowledgements}

The author would like to thank both his PhD advisors Jean Fasel and Andreas Rosenschon for many helpful discussions, for their encouragement and for their support. Furthermore, the author would like to thank Ravi Rao for his very helpful comments on elementary and symplectic orbits of unimodular rows and Anand Sawant for many helpful discussions on motivic homotopy theory. The author would also like to thank Aravind Asok for his very helpful comments on the generalized Serre conjecture on algebraic vector bundles. Moreover, he would like to thank the anonymous referee for suggesting changes which greatly improved the exposition of the paper. Finally, the author would like to thank Marc Levine and Marco Schlichting for their encouragement and for their interest in this work. The author was funded by a grant of the DFG priority program 1786 "Homotopy theory and algebraic geometry".

\section{Preliminaries}\label{Preliminaries}\label{2}

Let $R$ be a commutative ring. For any finitely generated projective $R$-module $P$, we denote by $Um (P)$ the set of epimorphisms $P \rightarrow R$ and by $Aut (P)$ the group of automorphisms of $P$; its group of automorphisms with determinant $1$ is denoted $SL (P)$. For any direct sum $P = \bigoplus_{i=1}^{n} P_{i}$ of finitely generated projective $R$-modules, we let $E (P_{1},...,P_{n})$ (or simply $E (P)$ if the decomposition is understood) be the subgroup of $Aut (P)$ generated by elementary automorphisms, i.e. automorphisms of the form $id_{P} + s$, where $s:P_{j} \rightarrow P_{i}$ is an $R$-linear map for $i \neq j$.\\
We denote by $Unim.El. (P)$ the set of unimodular elements of $P$. Note that $Aut (P)$ and hence any subgroup of $Aut (P)$ act on the right on $Um (P)$ and on the left on $Unim.El. (P)$.\\
If $P = R^{n}$, we naturally identify $Um (P)$ with the set $Um_{n} (R)$ of unimodular rows of length $n$ and $Unim.El. (P)$ with the set $Um_{n}^{t} (R)$ of unimodular columns of length $n$. We also identify $Aut (P)$, $SL (P)$ and $E (P)$ with $GL_{n} (R)$, $SL_{n} (R)$ and $E_{n} (R)$ in this case.

\subsection{Classification of stably isomorphic oriented projective modules}\label{Classification of stably isomorphic oriented projective modules}\label{2.1}

Again, let $R$ be a commutative ring. For any $r \in \mathbb{N}$, let $\mathcal{V}_{r} (R)$ be the set of isomorphism classes of projective modules of constant rank $r$. Furthermore, we denote by $\mathcal{V}_{r}^{o} (R)$ the set of isomorphism classes of oriented projective modules of rank $r$, i.e. isomorphism classes of pairs $(P, \theta)$, where $P$ is projective of constant rank $r$ and $\theta: \det (P) \xrightarrow{\cong} R$ is an isomorphism. An isomorphism between two such pairs $(P, \theta)$ and $(P', \theta')$ is an isomorphism $k: P \xrightarrow{\cong} P'$ such that $\theta = \theta' \circ \det(k)$. Note that if $(P, \theta)$ is an oriented projective module of rank $r$, then there is an induced orientation on $P \oplus R$ given by the composite $\theta^{+}: \det (P \oplus R) \cong \det (P) \xrightarrow{\theta} R$.\\
We now consider the stabilization maps

\begin{center}
$\phi_{r}^{o} : \mathcal{V}_{r}^{o} (R) \rightarrow \mathcal{V}_{r+1}^{o} (R), [(P, \theta)] \mapsto [(P \oplus R, \theta^{+})]$
\end{center}
 
from isomorphism classes of oriented projective modules of rank $r$ to isomorphism classes of oriented projective modules of rank $r+1$. We fix an oriented projective module $(P \oplus R, \theta^{+})$ representing an element of ${\mathcal{V}}_{n+1}^{o} (R)$ in the image of this map.\\
An element $[(P', \theta')]$ of ${\mathcal{V}}_{n}^{o} (R)$ lies in the fiber over $[(P \oplus R, \theta^{+})]$ if and only if there is an isomorphism $i: P' \oplus R \xrightarrow{\cong} P \oplus R$ such that $\theta^{+} \circ \det (i) = \theta'^{+}$. Any such isomorphism yields an element of $Um (P \oplus R)$ given by the composite 

\begin{center}
$a (i): P \oplus R \xrightarrow{{i}^{-1}} P' \oplus R \xrightarrow{\pi_R} R$.
\end{center}

If one chooses another oriented projective module $(P'', \theta'')$ representing the isomorphism class of $(P', \theta')$ and any isomorphism $j: P'' \oplus R \xrightarrow{\cong} P \oplus R$ with $\theta''^{+} = \theta^{+} \circ \det(j)$, the resulting element $a (j)$ of $Um (P \oplus R)$ still lies in the same orbit of $Um (P \oplus R)/SL (P \oplus R)$:
for if we choose an isomorphism $k: P' \xrightarrow{\cong} P''$ with $\theta' = \theta'' \circ \det(k)$, then ${j} {(k \oplus id_{R})} {i}^{-1} \in SL (P \oplus R)$ and we have an equality

\begin{center}
$a (i) = a (j) \circ ({j} {(k \oplus id_{R})} {i}^{-1})$.
\end{center}

Thus, we obtain a well-defined map

\begin{center}
${\phi_{r}^{o}}^{-1} ([(P \oplus R, \theta^{+})]) \rightarrow Um (P \oplus R)/SL (P \oplus R)$.
\end{center}

Conversely, any element $a \in Um (P \oplus R)$ gives an element of ${\mathcal{V}}_{r}^{o} (R)$ lying over $[(P \oplus R, \theta^{+})]$: If we let $P' = \ker (a)$, then the short exact sequence

\begin{center}
$0 \rightarrow P' \rightarrow P \oplus R \xrightarrow{a} R \rightarrow 0$
\end{center}

is split and any section $s$ of $a$ induces an isomorphism $i: P' \oplus R \xrightarrow{\cong} P \oplus R$. The induced isomorphism $\det(i): \det (P' \oplus R) \xrightarrow{\cong} \det (P \oplus R)$ does not depend on the section $s$; hence we can canonically define an orientation $\theta'$ on $P'$ given by the composite

\begin{center}
$\det(P') \cong \det (P' \oplus R) \xrightarrow{\det(i)} \det (P \oplus R) \xrightarrow{\theta^{+}} R$.
\end{center}

Then $[(P', \theta')] \in {\phi_{r}^{o}}^{-1} ([(P \oplus R, \theta^{+})])$. Note that this assignment only depends on the class of $a$ in $Um (P \oplus R)/SL (P \oplus R)$.\\
Thus, we also obtain a well-defined map

\begin{center}
$Um (P \oplus R)/SL(P \oplus R) \rightarrow {\phi_{r}^{o}}^{-1} ([(P \oplus R, \theta^{+})])$.
\end{center}

One can check that the maps ${\phi_{r}^{o}}^{-1} ([(P \oplus R, \theta^{+})]) \rightarrow Um (P \oplus R)/SL (P \oplus R)$ and $Um (P \oplus R)/SL(P \oplus R) \rightarrow {\phi_{r}^{o}}^{-1} ([(P \oplus R, \theta^{+})])$ are inverse to each other. Note that $[(P, \theta)]$ corresponds to the class represented by the canonical projection $\pi_{R} : P \oplus R \rightarrow R$ under these bijections. Altogether, we have a pointed bijection between the sets $Um (P \oplus R)/SL(P \oplus R)$ and ${\phi_{r}^{o}}^{-1} ([(P \oplus R, \theta^{+})])$ equipped with $[\pi_{R}]$ and $[(P, \theta)]$ as basepoints respectively.\\
Similarly, we can consider the stabilization maps

\begin{center}
$\phi_{r} : \mathcal{V}_{r} (R) \rightarrow \mathcal{V}_{r+1} (R), [P] \mapsto [P \oplus R]$
\end{center}

and deduce a pointed bijection between the sets $Um (P \oplus R)/Aut(P \oplus R)$ and $\phi_{r}^{-1} ([P \oplus R])$ equipped with $[\pi_{R}]$ and $[P]$ as basepoints respectively.\\
For any oriented projective module $(P, \theta)$ of rank $r$ as above, the canonical projection $Um (P \oplus R)/SL (P \oplus R) \rightarrow Um (P \oplus R)/Aut (P \oplus R)$ then corresponds to the map ${\phi_{r}^{o}}^{-1} ([(P \oplus R, \theta^{+})]) \rightarrow \phi_{r}^{-1} ([P \oplus R])$ forgetting the orientation of $P$.

\subsection{The groups $W'_G (R)$}\label{The groups $W'_G (R)$}\label{2.2}

Let R be a commutative ring. Moreover, let $G$ be any group such that $E (R) \subset G \subset SL (R)$. For any $n \in \mathbb{N}$, we denote by $A_{2n} (R)$ the set of alternating invertible matrices of rank $2n$. We inductively define an element $\psi_{2n} \in  A_{2n} (R)$ by setting
 
\begin{center}
$\psi_2 =
\begin{pmatrix}
0 & 1 \\
- 1 & 0
\end{pmatrix}
$
\end{center}
 
\noindent and $\psi_{2n+2} = \psi_{2n} \perp \psi_2$. For any $m < n$, there is an embedding of $A_{2m} (R)$ into $A_{2n} (R)$ given by $M \mapsto M \perp \psi_{2n-2m}$. We denote by $A (R)$ the direct limit of the sets $A_{2n} (R)$ under these embeddings. Two alternating invertible matrices $M \in A_{2m} (R)$ and $N \in A_{2n} (R)$ are called $G$-equivalent, $M \sim_G N$, if there is an integer $s \in \mathbb{N}$ and a matrix $E \in SL_{2n+2m+2s} (R) \cap G$ such that

\begin{center}
$M \perp \psi_{2n+2s} = E^{t} (N \perp \psi_{2m+2s}) E$.
\end{center}

The set of equivalence classes $A (R)/{\sim_G}$ is denoted $W'_G (R)$. Since
 
\begin{center}
$
\begin{pmatrix}
0 & id_{s} \\
id_{r} & 0
\end{pmatrix}
\in E_{r+s} (R)$
\end{center}
 
for even $rs$, it follows that the orthogonal sum equips $W'_G (R)$ with the structure of an abelian monoid. As it is shown in \cite{SV}, this abelian monoid is actually an abelian group. An inverse for an element of $W'_G (R)$ represented by a matrix $N \in A_{2n} (R)$ is given by the element represented by the matrix $\sigma_{2n} N^{-1} \sigma_{2n}$, where the matrices $\sigma_{2n}$ are inductively defined by
 
\begin{center}
$\sigma_2 =
\begin{pmatrix}
0 & 1 \\
1 & 0
\end{pmatrix}
$
\end{center}
 
\noindent and $\sigma_{2n+2} = \sigma_{2n} \perp \sigma_2$.
Now recall that one can assign to any alternating invertible matrix $M$ an element $\mathit{Pf} (M)$ of $R^{\times}$ called the Pfaffian of $M$. The Pfaffian satisfies the following formulas:
 
\begin{itemize}
\item $\mathit{Pf}(M \perp N) = \mathit{Pf} (M) \mathit{Pf}(N)$ for all $M \in A_{2m} (R)$ and $N \in A_{2n} (R)$;
\item $\mathit{Pf}(G^{t} N G) = \det (G) \mathit{Pf}(N)$ for all $G \in GL_{2n} (R)$ and $N \in A_{2n} (R)$;
\item ${\mathit{Pf} (N)}^{2} = \det (N)$ for all $N \in A_{2n} (R)$;
\item $\mathit{Pf} (\psi_{2n}) = 1$ for all $n \in \mathbb{N}$.
\end{itemize}

\noindent Therefore the Pfaffian determines a group homomorphism $\mathit{Pf}: W'_G (R) \rightarrow R^{\times}$; its kernel is denoted $W_G (R)$. Note that if $G = E (R)$, we recover the definition of the so-called elementary symplectic Witt group $W_E (R)$ of $R$. If $G = SL (R)$, we denote $W_G (R)$ simply by $W_{SL} (R)$.
 
\subsection{The group $V_{SL}(R)$}\label{The group $V_{SL}(R)$}\label{2.3}

Again, let $R$ be a commutative ring. Consider the set of triples $(P, g, f)$, where $P$ is a finitely generated projective $R$-module and $f,g$ are alternating isomorphisms (cp. \cite[Section 2.A]{Sy}). Two such triples $(P, f_0, f_1)$ and $(P', f'_0, f'_1)$ are called isometric if there is an isomorphism $h: P \rightarrow P^{'}$ such that $f_{i} = h^{\vee} f'_{i} h$ for $i=0,1$. We denote by $[P, g, f]$ the isometry class of the triple $(P, g, f)$.\\
Let $V (R)$ be the quotient of the free abelian group on isometry classes of triples as above modulo the subgroup generated by the relations
 
\begin{itemize}
\item $[P \oplus P', g \perp g', f \perp f'] = [P, g, f] + [P', g', f']$ for alternating isomorphisms $f,g$ on $P$ and $f',g'$ on $P'$;
\item $[P, f_0, f_1] + [P, f_1, f_2] = [P, f_0, f_2]$ for alternating isomorphisms $f_{0}, f_{1}$ and $f_{2}$ on $P$.
\end{itemize}

Note that these relations yield the useful identities
 
\begin{itemize}
\item $[P, f, f] = 0$ in $V (R)$ for any alternating isomorphism $f$ on $P$;
\item $[P, g, f] = - [P, f, g]$ in $V (R)$ for alternating isomorphisms $f,g$ on $P$;
\item $[P, g, {\beta}^{\vee} {\alpha}^{\vee} f \alpha \beta] = [P, f, {\alpha}^{\vee} f \alpha] + [P, g, {\beta}^{\vee} f \beta]$ in $V (R)$ for all automorphisms $\alpha, \beta$ of $P$ and alternating isomorphisms $f,g$ on $P$.
\end{itemize}

We may also restrict this construction to free $R$-modules of finite rank. The corresponding group will be denoted $V_{\mathit{free}} (R)$. Note that there is an obvious group homomorphism $V_{\mathit{free}} (R) \rightarrow V (R)$. Actually, this map is an isomorphism:\\
Recall that for any finitely generated projective $R$-module $P$ there is a standard alternating isomorphism

\begin{center}
$H_P =
\begin{pmatrix}
0 & id_{P^{\vee}} \\
-can & 0
\end{pmatrix}
: P \oplus P^{\vee} \rightarrow P^{\vee} \oplus P^{\vee\vee}$
\end{center}

on $P \oplus P^{\vee}$ called the hyperbolic isomorphism on $P$.\\ 
Now let $(P, g, f)$ be a triple as above. Since $P$ is a finitely generated projective $R$-module, there is another $R$-module $Q$ such that $P \oplus Q \cong R^{n}$ for some $n \in \mathbb{N}$. In particular, $P \oplus P^{\vee} \oplus Q \oplus Q^{\vee}$ is free of rank $2n$. Therefore the triple
 
\begin{center}
$(P \oplus P^{\vee} \oplus Q \oplus Q^{\vee}, g \perp can\; g^{-1} \perp H_Q, f \perp can\; g^{-1} \perp H_Q)$
\end{center}
 
represents an element of $V_{\mathit{free}} (R)$. It can be checked that this assignment descends to a well-defined group homomorphism
 
\begin{center}
$V (R) \rightarrow V_{\mathit{free}} (R)$.
\end{center}
 
By construction, this homomorphism is inverse to the canonical morphism $V_{\mathit{free}} (R) \rightarrow V (R)$. Thus, $V_{\mathit{free}} (R) \cong V (R)$.\\\\
The group $V_{\mathit{free}} (R)$ can be seen to be isomorphic to $W'_E (R)$ as follows: If $M \in A_{2m} (R)$ represents an element of $W'_E (R)$, then we assign to it the class in $V_{\mathit{free}} (R)$ represented by $[R^{2m}, \psi_{2m}, M]$. This assignment descends to a well-defined homomorphism $\nu: W'_E (R) \rightarrow V_{\mathit{free}} (R)$.\\
Now let us describe the inverse $\xi: V_{\mathit{free}} (R) \rightarrow W'_{E} (R)$ to this homomorphism. Let $(L, g, f)$ be a triple with $L$ free and $g,f$ alternating isomorphisms on $L$. We can choose an isomorphism $\alpha: R^{2n} \xrightarrow{\cong} L$ and consider the alternating isomorphism
 
\begin{center}
$({\alpha}^{t} f \alpha) \perp \sigma_{2n} {({\alpha}^{t} g \alpha)}^{-1} \sigma_{2n}^{\vee}: R^{2n} \oplus {(R^{2n})}^{\vee} \rightarrow {(R^{2n})}^{\vee} \oplus R^{2n}$.
\end{center}
 
With respect to the standard basis of $R^{2n}$ and its dual basis of ${(R^{2n})}^{\vee}$, we may interpret this alternating isomorphism as an element of $A_{4n} (R)$ and consider its class $\xi ([L, g, f])$ in $W'_E (R)$. It is proven in \cite{FRS} that this assignment induces a well-defined homomorphism $\xi: V_{\mathit{free}} (R) \rightarrow W'_E (R)$. By construction, $\nu$ and $\xi$ are obviously inverse to each other and therefore identify $W'_E (R)$ with $V_{\mathit{free}} (R)$. From now on, we denote by $\tilde{V} (R)$ the subgroup of $V (R)$ corresponding to $W_E (R)$ under the isomorphisms $V (R) \cong V_{\mathit{free}} (R) \cong W'_E (R)$.\\\\
In view of the previous paragraph, we obtain the following new presentation of the group $W'_{SL} (R)$:
Let $V_{SL} (R)$ be the quotient of the free abelian group on isometry classes of triples $(P, g, f)$ modulo the subgroup generated by the relations
 
\begin{itemize}
\item $[P \oplus P', g \perp g', f \perp f'] = [P, g, f] + [P', g', f']$ for alternating isomorphisms $f,g$ on $P$ and $f',g'$ on $P'$;
\item $[P, f_0, f_1] + [P, f_1, f_2] = [P, f_0, f_2]$ for alternating isomorphisms $f_{0}, f_{1}, f_{2}$ on $P$;
\item $[P, g, f] = [P, g, \varphi^{\vee} f \varphi]$ for alternating isomorphisms $g, f$ on $P$ and $\varphi \in SL (P)$.
\end{itemize}

Then $V_{SL} (R) \cong W'_{SL} (R)$. Automatically, there is a canonical epimorphism $\tilde{V} (R) \rightarrow \tilde{V}_{SL} (R)$ corresponding to the map $W_E (R) \rightarrow W_{SL} (R)$. We denote by $\tilde{V}_{SL} (R)$ the subgroup of $V_{SL} (R)$ corresponding to the group $W_{SL} (R)$. Again, there is a canonical epimorphism $\tilde{V} (R) \rightarrow \tilde{V}_{SL} (R)$ corresponding to the map $W_E (R) \rightarrow W_{SL} (R)$.

\begin{Lem}\label{L2.1}
If $[P, \chi, \chi_{1}] = [P, \chi, \chi_{2}] \in V_{SL} (R)$ for non-degenerate alternating forms $\chi$, $\chi_{1}$ and $\chi_{2}$ on a finitely generated projective $R$-module $P$, then we have an equality ${\alpha}^{t} (\chi_{1} \perp \psi_{2n}) \alpha = \chi_{2} \perp \psi_{2n}$ for some $n \in \mathbb{N}$ and some automorphism $\alpha \in SL (P \oplus R^{2n})$.
\end{Lem}

\begin{proof}
The equality $[P, \chi, \chi_{1}] = [P, \chi, \chi_{2}]$ means that $[P, \chi_{1}, \chi_{2}] = 0$. By \cite[Lemma 4.8]{Sy}, it follows that there is a finitely generated projective $R$-module $P_{1}$ with a non-degenerate alternating form $\chi'$ on $P_{1}$ and, moreover, with an isomorphism $\tau: R^{2m} \xrightarrow{\cong} P \oplus P_{1}$ such that ${\tau}^{t} (\chi_{1} \perp \chi') \tau = \psi_{2m}$. In particular, one has $0 = [P, \chi_{1}, \chi_{2}] = [R^{2m}, \psi_{2m}, {\tau}^{t}(\chi_{2} \perp \chi')\tau] \in \tilde{V}_{SL} (R)$. Therefore the class of ${\tau}^{t}(\chi_{2} \perp \chi')\tau$ in $W'_{SL} (R)$ is trivial and hence there exist $u \geq 1$ and $\zeta \in SL (R^{2m+2u})$ such that ${\zeta}^{t} (({\tau}^{t} (\chi_{2} \perp \chi') \tau) \perp \psi_{2u}) \zeta = \psi_{2m+2u}$.\\
Again by \cite[Lemma 4.8]{Sy}, there exists a finitely generated projective $R$-module $P_{2}$ with a non-degenerate alternating form $\chi''$ on $P_{2}$ and with an isomorphism $\beta: R^{2n} \xrightarrow{\cong} P_{1} \oplus R^{2u} \oplus P_{2}$ such that ${\beta}^{t} (\chi' \perp \psi_{2u} \perp \chi'') \beta = \psi_{2n}$.\\
But then the composite

\begin{center}
$\alpha = (id_{P} \oplus {\beta}^{-1})(\tau \oplus id_{R^{2u}} \oplus id_{P_{2}})({\zeta}^{-1} \oplus id_{P_{2}})({\tau}^{-1} \oplus id_{R^{2u}} \oplus id_{P_{2}})(id_{P} \oplus \beta)$
\end{center}

is an isometry from $\chi_{1} \perp \psi_{2n}$ to $\chi_{2} \perp \psi_{2n}$ and clearly has determinant $1$. This proves the lemma.
\end{proof}

\subsection{Motivic homotopy theory}\label{Motivic homotopy theory}\label{2.4}

In this section, we give a brief introduction to motivic homotopy theory. The main use of motivic homotopy theory in this paper will take place in the proof of Theorem \ref{T3.17}, which concerns symplectic orbits of unimodular rows; we will mainly use the identification $Um_{n} (R)/E_{n} (R) = [Spec (R), \mathbb{A}^{n}\setminus 0]_{\mathbb{A}^{1}}$ for $n \geq 3$ and any smooth affine algebra $R$ over a base field $k$ as well as the theory of fiber sequences and Suslin matrices, which we will explain in this section.\\
So let $k$ be a field and let $Sm_{k}$ be the category of smooth separated schemes of finite type over $k$. Then let $Spc_{k} = \Delta^{op} Shv_{Nis} (Sm_{k})$ (resp. $Spc_{k,\bullet}$) be the category of (pointed) simplicial Nisnevich sheaves on $Sm_{k}$. We write $\mathcal{H}_{s} (k)$ (resp. $\mathcal{H}_{s,\bullet} (k)$) for the (pointed) Nisnevich simplicial homotopy category which can be obtained as the homotopy category of the injective local model structure on $Spc_{k}$ (resp. $Spc_{k,\bullet}$). Furthermore, we write $\mathcal{H} (k)$ (resp. $\mathcal{H}_{\bullet} (k)$) for the $\mathbb{A}^{1}$-homotopy category, which can be obtained as a Bousfield localization of $\mathcal{H}_{s} (k)$ (resp. $\mathcal{H}_{s,\bullet} (k)$); see e.g. \cite{MV} for more details. Objects of $Spc_{k}$ (resp. $Spc_{k,\bullet}$) will be referred to as (pointed) spaces.\\
For two spaces $\mathcal{X}$ and $\mathcal{Y}$, we denote by $[\mathcal{X},\mathcal{Y}]_{\mathbb{A}^{1}} = Hom_{\mathcal{H} (k)} (\mathcal{X},\mathcal{Y})$ the set of morphisms from $\mathcal{X}$ to $\mathcal{Y}$ in $\mathcal{H} (k)$; similarly, for two pointed spaces $(\mathcal{X},x)$ and $(\mathcal{Y},y)$, we denote by $[(\mathcal{X},x), (\mathcal{Y},y)]_{\mathbb{A}^{1},\bullet} = Hom_{\mathcal{H}_{\bullet} (k)} ((\mathcal{X},x), (\mathcal{Y},y))$ the set of morphisms from $(\mathcal{X},x)$ to $(\mathcal{Y},y)$ in $\mathcal{H}_{\bullet} (k)$. Sometimes we will omit the basepoints from the notation.\\
Just as in classical topology, there are simplicial suspension and loop space functors $\Sigma_{s},\Omega_{s}: Spc_{k, \bullet} \rightarrow Spc_{k, \bullet}$, which form an adjoint Quillen pair of functors. The right-derived functor of $\Omega_{s}$ will be denoted $\mathcal{R}\Omega_{s}$. We denote by $\Sigma_{s}^{n}$ and $\Omega_{s}^{n}$ the iterated suspension and loop space functors for any $n \in \mathbb{N}$. For any pointed space $(\mathcal{X},x)$, its simplicial suspension $\Sigma_{s} (\mathcal{X},x) = S^{1} \wedge (\mathcal{X},x)$ has the structure of an $h$-cogroup in $\mathcal{H}_{\bullet} (k)$ (cp. \cite[Definition 2.2.7]{A} or \cite[Section 6.1]{Ho}); in particular, for any pointed space $(\mathcal{Y},y)$, there is a natural group structure on the set $[\Sigma_{s}(\mathcal{X},x), (\mathcal{Y},y)]_{\mathbb{A}^{1}, \bullet}$ induced by the $h$-cogroup structure of $\Sigma_{s} (\mathcal{X},x)$. For any pointed space $(\mathcal{Y},y)$, the space $\mathcal{R}\Omega_{s}(\mathcal{Y},y)$ has the structure of an $h$-group (or grouplike $H$-space in some literature) in $\mathcal{H}_{\bullet} (k)$ and hence the set $[(\mathcal{X},x), \mathcal{R}\Omega_{s}(\mathcal{Y},y)]_{\mathbb{A}^{1},\bullet}$ has a natural group structure for any pointed space $(\mathcal{X},x)$ induced by the $h$-group structure of $\mathcal{R}\Omega_{s}(\mathcal{Y},y)$.\\
Furthermore, the functor $Spc_{k} \rightarrow Spc_{k,\bullet}, \mathcal{X} \mapsto X_{+} = X \sqcup \ast$ and the forgetful functor $Spc_{k,\bullet} \rightarrow Spc_{k}$ form a Quillen pair which will be tacitly used in some proofs of this paper in order to force some spaces to have a basepoint. If $(\mathcal{X},x)$ is a pointed space and $i \geq 0$, then we let $\pi_{i}^{\mathbb{A}^{1}} (\mathcal{X},x)$ be the Nisnevich sheaf associated with the presheaf $U \mapsto [\Sigma_{s}^{i} U_{+}, (\mathcal{X},x)]_{\mathbb{A}^{1},\bullet}$.\\
Recall that in any pointed model category, i.e. in any model category whose initial and terminal object are isomorphic, there exists the notion of fiber sequences $(\mathcal{F},f) \hookrightarrow (\mathcal{E},e) \rightarrow (\mathcal{B},b)$ (cp. \cite[Section 6.2]{Ho}). Since $Spc_{k,\bullet}$ is a pointed model category with its $\mathbb{A}^{1}$-model structure, this notion in particular exists in motivic homotopy theory. Analogous to the situation in classical topology, such fiber sequences give rise to long exact sequences of groups and pointed sets of the form

\begin{center}
$...\rightarrow  [\mathcal{X}, \mathcal{R}\Omega_{s}(\mathcal{B},b)]_{\mathbb{A}^{1},\bullet} \rightarrow [\mathcal{X}, (\mathcal{F},f)]_{\mathbb{A}^{1},\bullet} \rightarrow [\mathcal{X}, (\mathcal{E},e)]_{\mathbb{A}^{1},\bullet} \rightarrow [\mathcal{X}, (\mathcal{B},b)]_{\mathbb{A}^{1},\bullet}$
\end{center}

for any pointed space $\mathcal{X}$ (see \cite[Section 6.5]{Ho}). For the purpose of this paper, we simply state the existence of the following $\mathbb{A}^{1}$-fiber sequences which follows from \cite[Section 5]{W}:

\begin{Thm}
Let $(X,x)$ be a pointed $k$-scheme. If $G = Sp_{2n},SL_{n},GL_{n}$ and $P \rightarrow X$ is a $G$-torsor, then there is an $\mathbb{A}^{1}$-fiber sequence of the form

\begin{center}
$G \hookrightarrow P \rightarrow X$.
\end{center}
\end{Thm}

As special cases of this theorem, we obtain $\mathbb{A}^{1}$-fiber sequences of the form

\begin{center}
$SL_{n} \hookrightarrow SL_{n+1} \rightarrow SL_{n+1}/SL_{n}$,\\
$Sp_{2n} \hookrightarrow SL_{2n} \rightarrow SL_{2n}/Sp_{2n}$,\\
$Sp_{2n} \hookrightarrow GL_{2n} \rightarrow GL_{2n}/Sp_{2n}$.
\end{center}

Let us describe the quotients $SL_{n}/SL_{n-1}$: For $n \geq 1$, the projection on the first column induces a morphism $SL_{n}/SL_{n-1} \rightarrow \mathbb{A}^{n}\setminus 0$ which is Zariski locally trivial with fibers isomorphic to $\mathbb{A}^{n-1}$ and hence an $\mathbb{A}^{1}$-weak equivalence.\\
For all $n \geq 1$, let $S_{2n-1} = k[x_{1},...,x_{n},y_{1},...,y_{n}]/\langle \sum_{i=1}^{n} x_{i}y_{i} - 1 \rangle$ and then let $Q_{2n-1} = Spec (S_{2n-1})$ be the smooth affine quadric hypersurfaces in $\mathbb{A}^{2n}$. The projection on the coefficients $x_{1},...,x_{n}$ induces a morphism of schemes $p_{2n-1}:Q_{2n-1} \rightarrow \mathbb{A}^{n}\setminus 0$ which is locally trivial with fibers isomorphic to $\mathbb{A}^{n-1}$ and hence an $\mathbb{A}^{1}$-weak equivalence. Thus, we have $\mathbb{A}^{1}$-weak equivalences

\begin{center}
$SL_{n}/SL_{n-1} \simeq_{\mathbb{A}^{1}} \mathbb{A}^{n}\setminus 0 \simeq_{\mathbb{A}^{1}} Q_{2n-1}$
\end{center}

for all $n \geq 1$. Note that these $\mathbb{A}^{1}$-weak equivalences are all pointed, if we equip $SL_{n}/SL_{n-1}$ with the identity matrix, $\mathbb{A}^{n}\setminus 0$ with $(1,0,..,0)$ and $Q_{2n-1}$ with $(1,0,..,0,1,0,..,0)$ as basepoints.\\
If $R$ is a smooth affine $k$-algebra and $X = Spec(R)$, then it is well-known that

\begin{center}
$\mathit{Um}_{n} (R) \cong Hom_{Sm_{k}} (X, \mathbb{A}^{n}\setminus 0)$
\end{center}

and

\begin{center}
$\{(a,b)|a,b \in \mathit{Um}_{n} (R), a b^{t} = 1\} = Hom_{Sm_{k}} (X, Q_{2n-1})$.
\end{center}

If $n \geq 3$, it follows from \cite[Remark 8.10]{Mo} and \cite[Theorem 2.1]{F} that

\begin{center}
$\mathit{Um}_{n} (R)/E_{n} (R) \cong [X, \mathbb{A}^{n}\setminus 0]_{\mathbb{A}^{1}}$.
\end{center}

In particular, if $m \geq 1$, then the orbit space $\mathit{Um}_{n} (S_{2m-1})/E_{n} (S_{2m-1})$ is just given by

\begin{center}
$[Q_{2m-1}, \mathbb{A}^{n}\setminus 0]_{\mathbb{A}^{1}} \cong [\mathbb{A}^{m}\setminus 0, \mathbb{A}^{n}\setminus 0]_{\mathbb{A}^{1}}$.
\end{center}

It is well-known that $\mathbb{A}^{m}\setminus 0$ is isomorphic to $\Sigma_{s}^{m-1} \mathbb{G}_{m}^{\wedge m}$ in $\mathcal{H}_{\bullet} (k)$ for all $m \geq 1$; therefore $\mathbb{A}^{m}\setminus 0$ inherits the structure of an $h$-cogroup in $\mathcal{H}_{\bullet} (k)$ for $m \geq 2$ (cp. \cite[Definition 2.2.7]{A} or \cite[Section 6.1]{Ho}). In particular, the orbit space $\mathit{Um}_{n} (S_{2m-1})/E_{n} (S_{2m-1})$ has a natural group structure for $m \geq 2$, $n \geq 3$.\\\\
Now let $R$ be a commutative ring, $n \geq 1$ and $a = (a_{1},...,a_{n}), b = (b_{1},...,b_{n})$ be row vectors of length $n$. In \cite{S2} Suslin inductively constructed matrices $\alpha_{n} (a,b)$ of size $2^{n-1}$ called Suslin matrices for all $n \geq 1$: For $n=1$, one simply sets $\alpha_{1} (a,b) = (a_{1})$; for $n \geq 2$, one sets $a'=(a_{2},...,a_{n}),b'=(b_{2},...,b_{n})$ and defines

\begin{center}
$\alpha_{n} (a,b) = \begin{pmatrix}
a_{1} {id}_{2^{n-2}} & \alpha_{n-1} (a',b') \\
-{\alpha_{n-1} (b',a')}^{t} & b_{1} {id}_{2^{n-2}}
\end{pmatrix}.$
\end{center}

In \cite[Lemma 5.1]{S2} Suslin proved that $\det (\alpha_{n} (a,b)) = {(a b^{t})}^{2^{n-2}}$ if $n \geq 2$; in particular, if $a = (a_{1},...,a_{n})$ is a unimodular row of length $n$ and $b=(b_{1},...,b_{n})$ defines a section of $a$, i.e. $a b^{t}=\sum_{i=1}^{n} a_{i} b_{i} = 1$, then $\alpha_{n} (a,b) \in SL_{2^{n-1}} (R)$.\\
Suslin originally introduced these matrices in order to show that for any unimodular row $a = (a_{1},a_{2},a_{3},...,a_{n})$ of length $n \geq 3$, the row of the form $a' = (a_{1},a_{2},a_{3},...,a_{n}^{(n-1)!})$ is completable to an invertible matrix. In fact, he proved that for any $a$ with section $b$ there exists an invertible $n \times n$-matrix $\beta (a,b)$ whose first row is $a'$ such that the classes of $\beta (a,b)$ and $\alpha_{n} (a,b)$ in $K_{1} (R)$ coincide (cp. \cite[Proposition 2.2 and Corollary 2.5]{S3}).\\
As explained in \cite{AF4}, one can in fact interpret Suslin's construction as a morphism of schemes: We let $Q_{2n-1} = Spec (k[x_{1},...,x_{n},y_{1},...,y_{n}]/\langle \sum_{i=1}^{n}x_{i}y_{i} -1 \rangle)$ as above. Then there exists a morphism $\alpha_{n}: Q_{2n-1} \rightarrow SL_{2^{n-1}}$ induced by $\alpha_{n} (x,y)$, where $x=(x_{1},...,x_{n})$ and $y=(y_{1},...,y_{n})$; if we equip $Q_{2n-1}$ with $(1,0,..,0,1,0,...,0)$ and $SL_{2^{n-1}}$ with the identity as basepoints, this morphism is pointed. Composing with the canonical map $SL_{2^{n-1}} \rightarrow SL$, we obtain a morphism $Q_{2n-1} \rightarrow SL$ which we also denote by $\alpha_{n}$. If $R$ is a smooth affine algebra over $k$, then the induced morphism

\begin{center}
$\mathit{Um}_{n} (R)/E_{n} (R) \cong [Spec(R), Q_{2n-1}]_{\mathbb{A}^{1}} \xrightarrow{{\alpha_{n}}_{\ast}} [Spec(R), SL]_{\mathbb{A}^{1}} \cong SK_{1} (R)$
\end{center}

takes the class of any $a \in \mathit{Um}_{n} (R)$ to the class of $\alpha_{n} (a,b)$ in $SK_{1} (R)$, where $b$ is any section of $a$.

\subsection{Grothendieck-Witt groups}\label{Grothendieck-Witt groups}\label{2.5}

In this section we first recall some basics about higher Grothendieck-Witt groups, which are a modern version of Hermitian K-theory. The general references of the modern theory are \cite{MS1}, \cite{MS2} and \cite{MS3}. At the end of this section, we will then use the Gersten-Grothendieck-Witt spectral sequence in order to give cohomological criteria for the $2$-divisibility of the groups $W_E (R)$ and $W_{SL} (R)$ whenever $R$ is a smooth affine algebra of dimension $4$ over an algebraically closed field $k$ with $char(k) \neq 2$.\\
Now let $X$ be a scheme with $\frac{1}{2} \in \Gamma (X, \mathcal{O}_{X})$ and let $\mathcal{L}$ be a line bundle on $X$. Then we consider the category $C^{b} (X)$ of bounded complexes of locally free coherent $\mathcal{O}_{X}$-modules. The category $C^{b} (X)$ inherits a natural structure of an exact category from the category of locally free coherent $\mathcal{O}_{X}$-modules by declaring $C'_{\bullet} \rightarrow C_{\bullet} \rightarrow C''_{\bullet}$ to be exact if and only if $C'_{n} \rightarrow C_{n} \rightarrow C''_{n}$ is exact for all $n$. The duality $Hom_{\mathcal{O}_{X}} (-, \mathcal{L})$ induces a duality $\#_{\mathcal{L}}$ on $C^{b} (X)$ in the sense of \cite[\S 2.3]{MS2} and the isomorphism $id \rightarrow Hom_{\mathcal{O}_{X}} (Hom_{\mathcal{O}_{X}} (-, \mathcal{L}), \mathcal{L})$ for locally free coherent $\mathcal{O}_{X}$-modules induces a natural isomorphism of functors $\varpi_{\mathcal{L}}: id \xrightarrow{\sim} \#_{\mathcal{L}}\#_{\mathcal{L}}$ on $C^{b} (X)$. Moreover, the translation functor $T: C^{b} (X) \rightarrow C^{b} (X)$ yields new dualities $\#_{\mathcal{L}}^{j} = T^{j} \#_{\mathcal{L}}$ and natural isomorphisms $\varpi_{\mathcal{L}}^{j} = {(-1)}^{j (j+1)/2} \varpi_{\mathcal{L}}$. We say that a morphism in $C^{b} (X)$ is a weak equivalence if and only if it is a quasi-isomorphism and we denote by $qis$ the class of quasi-isomorphisms. For all $j$, the quadruple $(C^{b} (X), qis, \#_{\mathcal{L}}^{j}, \varpi_{\mathcal{L}}^{j})$ is an exact category with weak equivalences and strong duality (cp. \cite[\S 2.3]{MS2}).\\
Following \cite{MS2}, one can associate a Grothendieck-Witt space $\mathcal{GW}$ to any exact category with weak equivalences and strong duality. The (higher) Grothendieck-Witt groups are then defined to be its homotopy groups:
 
\begin{Def}
For any $i \geq 0$, we let $\mathcal{GW} (C^{b} (X), qis, \#_{\mathcal{L}}^{j}, \varpi_{\mathcal{L}}^{j})$ denote the Grothendieck-Witt space associated to the quadruple $(C^{b} (X), qis, \#_{\mathcal{L}}^{j}, \varpi_{\mathcal{L}}^{j})$ as above. Then we define ${GW}_{i}^{j} (X, \mathcal{L}) = \pi_{i} \mathcal{GW} (C^{b} (X), qis, \#_{\mathcal{L}}^{j}, \varpi_{\mathcal{L}}^{j})$. If $\mathcal{L} = \mathcal{O}_{X}$, we also denote $GW_{i}^{j} (X, \mathcal{O}_{X})$ by $GW_{i}^{j} (X)$. Furthermore, if $X = Spec (R)$, we simply denote $GW_{i}^{j} (X, \mathcal{L})$ or $GW_{i}^{j} (X)$ by $GW_{i}^{j} (R, \mathcal{L})$ or $GW_{i}^{j} (R)$ respectively.
\end{Def}
 
The groups $GW_{i}^{j} (X, \mathcal{L})$ are $4$-periodic in $j$. If we let $X = Spec(R)$ be an affine scheme, the groups $GW_{i}^{j} (X)$ coincide with Hermitian K-theory and U-theory as defined by Karoubi (cp. \cite{MK1} and \cite{MK2}), because $\frac{1}{2} \in \Gamma (X, \mathcal{O}_{X})$ by our assumption (cp. \cite[Remark 4.13]{MS1} and \cite[Theorems 6.1-2]{MS3}).\\
In particular, we have isomorphisms $K_{i}O (R) = GW_{i}^{0} (R)$, $_{-1}U_{i} (R) = GW_{i}^{1} (R)$, $K_{i}Sp (R) = GW_{i}^{2} (R)$ and $U_{i} (R) = GW_{i}^{3} (R)$.\\
For all $i,j \geq 0$, there are forgetful homomorphisms $f_{i,j}: GW_{i}^{j} (X) \rightarrow K_{i} (X)$, hyperbolic homomorphisms $H_{i,j}: K_{i} (X) \rightarrow GW_{i}^{j} (X)$ and also boundary homomorphisms $\eta: GW_{i+1}^{j+1} (X) \rightarrow GW_{i}^{j} (X)$ which are connected by means of the exact sequence called Karoubi periodicity sequence of the form

\begin{center}
$K_{i} (X) \xrightarrow{H_{i,j}} GW_{i}^{j} (X) \xrightarrow{\eta} GW_{i-1}^{j-1} (X) \xrightarrow{f_{i-1,j-1}} K_{i-1} (X) \xrightarrow{H_{i-1,j}} GW_{i-1}^{j} (X)$.
\end{center}

Now let $X = Spec (R)$ be affine. The group of our interest is $GW_{1}^{3} (X) = U_{1} (R)$. As a matter of fact, it is proved in \cite[Theorem 4.4]{FRS} that there is a natural isomorphism $GW_{1}^{3} (R) \cong W'_E (R)$. One of the main tools to compute the group $GW_{1}^{3} (X)$ is the Karoubi periodicity sequence mentioned above. By means of the identification $GW_{1}^{3} (X) \cong W'_E (R)$, this yields an exact sequence of the form

\begin{center}
$K_{1}{Sp} (R) \xrightarrow{f_{1,2}} K_{1} (R) \xrightarrow{H_{1,3}} W'_{E} (R) \xrightarrow{\eta} K_{0}{Sp} (R) \xrightarrow{f_{0,2}} K_{0} (R)$.
\end{center}

The homomorphisms in this sequence can be explicitly described as follows: The forgetful homomorphisms $K_{1}{Sp} (R) \xrightarrow{f_{1,2}} K_{1} (R)$ and $K_{0}{Sp} (R) \xrightarrow{f_{0,2}} K_{0} (R)$ are induced by the obvious inclusions $Sp_{2n} (R) \rightarrow GL_{2n} (R)$ and the assignment $(P, \varphi) \mapsto P$ for any skew-symmetric space $(P, \varphi)$ respectively. Moreover, the hyperbolic map $K_{1} (R) \xrightarrow{H_{1,3}} W'_{E} (R)$ is induced by the assignment $M \mapsto M^{t} \psi_{2n} M$ for all $M \in GL_{2n} (R)$. Finally, the homomorphism $W'_{E} (R) \xrightarrow{\eta} K_{0}{Sp} (R)$ is induced by the assignment $M \mapsto [R^{2n}, M] - [R^{2n},\psi_{2n}]$ for all $M \in A_{2n} (R)$.\\
As the image of $K_{1}{Sp} (R)$ under $f_{1,2}$ in $K_{1} (R)$ lies in $SK_{1} (R)$, one can rewrite the sequence above as

\begin{center}
$K_{1}{Sp} (R) \xrightarrow{f_{1,2}} SK_{1} (R) \xrightarrow{H_{1,3}} W_{E} (R) \xrightarrow{\eta} K_{0}{Sp} (R) \xrightarrow{f_{0,2}} K_{0} (R)$.
\end{center}

If we restrict ourselves to smooth $k$-schemes over a perfect field $k$ of $char (k) \neq 2$, then it is known (cp. \cite[Theorem 3.1]{JH}) that Grothendieck-Witt groups are representable in the (pointed) $\mathbb{A}^{1}$-homotopy category $\mathcal{H}_{\bullet} (k)$ as defined by Morel and Voevodsky. As a matter of fact, if we let $X$ be a smooth $k$-scheme over a perfect field $k$, it is shown that there are pointed spaces $\mathcal{GW}^{j}$ such that
 
\begin{center}
$[\Sigma_{s}^{i} X_{+}, \mathcal{GW}^{j}]_{\mathbb{A}^{1}, \bullet} \cong GW_{i}^{j} (X)$.
\end{center}

Let us make these spaces more explicit: We consider for all $n \in \mathbb{N}$ the closed embeddings $GL_{n} \rightarrow O_{2n}$ and $GL_{n} \rightarrow Sp_{2n}$ given by

\begin{center}
$M \mapsto
\begin{pmatrix}
M & 0 \\
0 & {(M^{-1})}^{t}
\end{pmatrix}
$.
\end{center}

These embeddings are compatible with the standard stabilization embeddings $GL_{n} \rightarrow GL_{n+1}$, $O_{2n} \rightarrow O_{2n+2}$ and $Sp_{2n} \rightarrow Sp_{2n+2}$. Taking direct limits over all $n$ with respect to the induced maps $O_{2n}/GL_{n} \rightarrow O_{2n+2}/GL_{n+1}$ and $Sp_{2n}/GL_{n} \rightarrow Sp_{2n+2}/GL_{n+1}$, we obtain spaces $O/GL$ and $Sp/GL$. Similarly, the natural inclusions $Sp_{2n} \rightarrow GL_{2n}$ are compatible with the standard stabilization embeddings and we obtain a space $GL/Sp = colim_{n}\; GL_{2n}/Sp_{2n}$. As proven in \cite[Theorems 8.2 and 8.4]{ST}, there are canonical $\mathbb{A}^{1}$-weak equivalences
 
\begin{center}
\begin{equation*}
   \mathcal{GW}^{j} \cong
   \begin{cases}
     \mathbb{Z} \times OGr & \text{if } j \equiv 0 \text{ mod } 4 \\
     Sp/GL & \text{if } j \equiv 1 \text{ mod } 4 \\
     \mathbb{Z} \times HGr & \text{if } j \equiv 2 \text{ mod } 4 \\
     O/GL & \text{if } j \equiv 3 \text{ mod } 4
   \end{cases}
\end{equation*}
\end{center}
 
and
 
\begin{center}
$\mathcal{R}\Omega_{s}^{1} O/GL \cong GL/Sp$,
\end{center}
 
where $OGr$ is an "infinite orthogonal Grassmannian" and $HGr$ is an "infinite symplectic Grassmannian". As a consequence of all this, there is an isomorphism $[X, GL/Sp]_{\mathbb{A}^{1}} \cong GW_{1}^{3} (X)$. If we let $A_{2n}$ denote the scheme of skew-symmetric invertible $2n \times 2n$-matrices for all $n \in \mathbb{N}$, then it is argued in \cite{AF4} that the morphisms of schemes $GL_{2n} \rightarrow A_{2n}$, $M \mapsto M^{t} \psi_{2n} M$ induce an isomorphism $GL/Sp \cong A$ of Nisnevich sheaves, where $A = colim_{n} A_{2n}$ (the transition maps are given by adding $\psi_{2}$). Altogether, we obtain a bijection $[X, A]_{\mathbb{A}^{1}} \cong GW_{1}^{3} (X)$; if $X = Spec (R)$ is affine, then $[X, A]_{\mathbb{A}^{1}}$ is precisely $A (R) /{\sim_{E(R)}} = W'_E (R)$.\\
Similarly, for any smooth $k$-scheme $X$, we let $GW_{1,red}^{3} (X) = [X, SL/Sp]_{\mathbb{A}^{1}}$ be the reduced Grothendieck-Witt group. If $X = Spec (R)$ is affine, then it follows as in the previous paragraph that $GW_{1,red}^{3} (X) = W_E (R)$.\\
It follows from these identifications that the $\mathbb{A}^{1}$-fiber sequences 

\begin{center}
$Sp \rightarrow GL \rightarrow GL/Sp$\\
$Sp \rightarrow SL \rightarrow SL/Sp$
\end{center}

induce the homomorphisms

\begin{center}
$K_{1}{Sp} (R) \xrightarrow{f} K_{1} (R) \xrightarrow{H} W'_E (R)$\\
$K_{1}{Sp} (R) \xrightarrow{f} SK_{1} (R) \xrightarrow{H} W_E (R)$
\end{center}

in the Karoubi periodicity sequence above.\\\\
We now introduce Grothendieck-Witt sheaves and study their cohomology. This will give cohomological obstructions to the $2$-divisibility of $W_E (R)$ and $W_{SL} (R)$ for any smooth affine fourfold over an algebraically closed field with characteristic $\neq 2$.\\
For this, we first fix a perfect base field $k$ with $char(k) \neq 2$. Recall that we have defined $\mathbb{A}^{1}$-homotopy sheaves $\pi_{i}^{\mathbb{A}^{1}} (\mathcal{X},x)$ for any pointed space $(\mathcal{X},x) \in Spc_{k,\bullet}$. As a special case, we define Grothendieck-Witt sheaves as follows:

\begin{Def}
For any $i \geq 0$, we set $\textbf{\textsc{GW}}_{i}^{j} = \pi_{i}^{\mathbb{A}^{1}} (\mathcal{GW}^{j})$.
\end{Def}

Now let $X = Spec (R)$ be a smooth affine $k$-scheme. The Karoubi periodicity sequence induces an exact sequence of sheaves

\begin{center}
$\textbf{K}_{4}^{Q} \xrightarrow{H_{4,3}} \textbf{GW}_{4}^{3} \xrightarrow{\eta} \textbf{GW}_{3}^{2} \xrightarrow{f_{3,2}} \textbf{K}_{3}^{Q}$,
\end{center}

where $\textbf{K}_{i}^{Q}$ denotes the Quillen $K$-theory sheaves for $i=3,4$. We denote by $\textbf{A}$ the image of $H_{4,3}$ and by $\textbf{B}$ the image of $\eta$ and obtain a short exact sequence

\begin{center}
$0 \rightarrow \textbf{A} \rightarrow \textbf{GW}_{4}^{3} \rightarrow \textbf{B} \rightarrow 0$
\end{center}

of sheaves. It follows from \cite[Lemma 4.11]{AF2} and from the computations in \cite[Section 3.6]{AF3} that the associated exact sequence of cohomology groups yields an exact sequence of the form

\begin{center}
$H^{3} (X, \textbf{K}_{4}^{Q}/2) \rightarrow H^{3} (X, \textbf{GW}_{4}^{3}) \rightarrow Ch^{3} (X) \rightarrow Ch^{4} (X) \rightarrow H^{4} (X, \textbf{GW}_{4}^{3}) \rightarrow 0$,
\end{center}

where $Ch^{i} (X) = CH^{i} (X)/2$ for $i=3,4$. Since $CH^{4} (X)$ is uniquely $2$-divisible for any smooth affine fourfold $X$ over an algebraically closed field (cp. \cite{Sr}), we obtain:

\begin{Prop}
If $X = Spec(R)$ is a smooth affine fourfold over an algebraically closed field $k$ with $char(k) \neq 2$, then there is an exact sequence $H^{3} (X, \textbf{\textsc{K}}_{4}^{Q}/2) \rightarrow H^{3} (X, \textbf{\textsc{GW}}_{4}^{3}) \rightarrow Ch^{3} (X) \rightarrow 0$.
\end{Prop}

In particular, if $H^{3} (X, \textbf{K}_{4}^{Q}/2)$ and $Ch^{3} (X)$ are trivial, then also $H^{3} (X, \textbf{GW}_{4}^{3})$ is trivial. In fact, one can prove the following statement:

\begin{Prop}\label{P2.6}
If $X = Spec(R)$ is a smooth affine fourfold over an algebraically closed field $k$ with $char(k) \neq 2$, then $H^{3} (X, \textbf{\textsc{K}}_{4}^{Q})$ is $2$-divisible and $H^{3} (X, \textbf{\textsc{K}}_{4}^{Q}/2) = 0$. In particular, $H^{3} (X, \textbf{\textsc{GW}}_{4}^{3})$ is $2$-divisible if and only if $CH^{3} (X)$ is $2$-divisible.
\end{Prop}

\begin{proof}
We let $2\textbf{K}_{4}^{Q}$ be the image and $\{2\}\textbf{K}_{4}^{Q}$ be the kernel of the morphism $\textbf{K}_{4}^{Q} \rightarrow \textbf{K}_{4}^{Q}$ induced by multiplication by $2$. Then we consider the two short exact sequences of sheaves

\begin{center}
$0 \rightarrow \{2\}\textbf{K}_{4}^{Q} \rightarrow \textbf{K}_{4}^{Q} \rightarrow 2\textbf{K}_{4}^{Q} \rightarrow 0$
\end{center}

and

\begin{center}
$0 \rightarrow 2\textbf{K}_{4}^{Q} \rightarrow \textbf{K}_{4}^{Q} \rightarrow \textbf{K}_{4}^{Q}/2 \rightarrow 0$.
\end{center}

The Gersten resolutions of $\{2\}\textbf{K}_{4}^{Q}$ and $\textbf{K}_{4}^{Q}/2$ are flasque resolutions of these sheaves and can therefore be used in order to compute their cohomology.\\
Since $K_{0} (F) = \mathbb{Z}$ for any field $F$, we have $H^{4} (X, \{2\}\textbf{K}_{4}^{Q}) = 0$. It follows that the map $H^{3} (X, \textbf{K}_{4}^{Q}) \rightarrow H^{3} (X, 2\textbf{K}_{4}^{Q})$ is surjective. As the composite

\begin{center}
$H^{3} (X, \textbf{K}_{4}^{Q}) \rightarrow H^{3} (X, 2\textbf{K}_{4}^{Q}) \rightarrow H^{3} (X, \textbf{K}_{4}^{Q})$
\end{center}

is multiplication by $2$, it thus suffices to prove that $H^{3} (X, \textbf{K}_{4}^{Q}/2) = 0$.\\
For any $q,m \in \mathbb{N}$, we let $\mathcal{H}^{q} (m)$ be the sheaf associated to the presheaf

\begin{center}
$U \mapsto H^{q}_{\acute{e}t} (U, \mu_{2}^{\otimes m})$.
\end{center}

Recall that the Bloch-Ogus spectral sequence (cp. \cite{BO}) converges to the \'etale cohomology groups $H^{\ast}_{\acute{e}t} (X, \mu_{2}^{\otimes m})$ and its terms on the second page are $H^{p}_{Zar} (X,\mathcal{H}^{q}(m))$. These groups can be computed via the Gersten complex

\begin{center}
$H^{q} (k(X),\mu_{2}^{\otimes m}) \xrightarrow{d_{0}} \oplus_{x_{1} \in X^{(1)}} H^{q-1} (k(x_{1}),\mu_{2}^{\otimes m-1}) \xrightarrow{d_{1}} ...$.
\end{center}

By \cite[\S 4.2, Proposition 11]{S}, one has $c.d.(k(x_{p})) \leq 4-p$ for any $x_{p} \in X^{(p)}$. Therefore it follows that $H^{p}_{Zar} (X,\mathcal{H}^{q}(m)) = 0$ for all $q \geq 5$; consequently, $H^{3} (X, \mathcal{H}^{4}(m)) = H^{7}_{\acute{e}t} (X, \mu_{2}^{\otimes m}) = 0$, because $X$ is affine.\\
Since $H^{3} (X, \mathcal{H}^{4}(4)) = H^{3} (X,\textbf{K}_{4}^{Q}/2)$ because of Voevodsky's resolution of the Milnor conjectures, this proves the result.
\end{proof}

In the remainder of this section, we will use the Gersten-Grothendieck-Witt spectral sequence in order to compute $W_E (S_{2n-1})$ for all $n$ divisible by $4$ and in order to find cohomological obstructions for the $2$-divisibility of $W_E (R)$ when $R$ is a smooth affine $k$-algebra of dimension $4$ and $k$ is algebraically closed. The Gersten-Grothendieck-Witt spectral sequence (cp. \cite[Theorem 25]{FS2}) is the analogue in the theory of higher Grothendieck-Witt groups of the classical Brown-Gersten-Quillen spectral sequence in algebraic $K$-theory.\\
Recall that if $X$ is a smooth $k$-scheme of dimension $d$, then the Gersten-Grothendieck-Witt spectral sequence $E(3)$ associated to $X$ has terms of the form

\begin{center}
\begin{equation*}
   E(3)_{1}^{p,q} \cong
   \begin{cases}
     \bigoplus_{x_{p} \in X^{(p)}} GW_{3-p-q}^{3-p} (k(x_{p}), \omega_{x_{p}}) & \text{if } 0 \leq p \leq d \text{ and } 3 \geq p+q\\
     0 & \text{else}
   \end{cases}
\end{equation*}
\end{center}

on the first page and converges to $GW_{3-\ast}^{3} (X)$. There is a filtration

\begin{center}
$0 = F_{d+1} \subset F_{d} \subset ... \subset F_{1} \subset GW_{1}^{3} (X) = F_{0}$
\end{center}

with $F_{p}/F_{p+1} \cong E(3)_{\infty}^{p,2-p}$ for all $p$. Furthermore, the terms $E(3)_{2}^{p,q}$ on the second page are isomorphic to $H^{p} (X, \textbf{GW}_{3-q}^{3})$ for $0 \leq p \leq d$ and $p+q \leq 3$. The group $F_{1}$ coincides with $GW_{1,red}^{3} (X) = [X, SL/Sp]_{\mathbb{A}^{1}}$. In particular, if $X = Spec (R)$ is affine, then it coincides with $W_E (R)$. Hence we can compute the group $W_E (R)$ via the limit terms $E(3)_{\infty}^{p,2-p}$.

\begin{Prop}
Let $n \in \mathbb{N}$ be divisible by $4$. Then there is an isomorphism $W_E (S_{2n-1}) \cong \mathbb{Z}/2\mathbb{Z}$.
\end{Prop}

\begin{proof}
We have identifications

\begin{center}
$W_E (S_{2n-1}) \cong [Q_{2n-1}, SL/Sp]_{\mathbb{A}^{1}} \cong [\mathbb{A}^{n}\setminus 0, SL/Sp]_{\mathbb{A}^{1}} = GW_{1,red}^{3} (\mathbb{A}^{n}\setminus 0)$.
\end{center}

We use the Gersten-Grothendieck-Witt spectral sequence $E (3)$ associated to $X = \mathbb{A}^{n} \setminus 0$ in order to compute $GW_{1,red}^{3} (\mathbb{A}^{n}\setminus 0)$. As indicated above, we have a filtration

\begin{center}
$0 = F_{n+1} \subset F_{n} \subset ... \subset GW_{1,red}^{3} (X) = F_{1} \subset GW_{1}^{3} (X) = F_{0}$
\end{center}

with $F_{p}/F_{p+1} \cong E(3)_{\infty}^{p,2-p}$ for all $p$.\\
Let us compute the limit terms $E(3)_{\infty}^{p,q}$. It is known that the terms $E(3)_{2}^{p,q}$ on the second page are precisely isomorphic to $H^{p} (X, \textbf{GW}_{3-q}^{3})$. Since $n$ is divisible by $4$, it follows from \cite[Lemma 4.5]{AF1} that

\begin{center}
\begin{equation*}
   E(3)_{2}^{p,q} \cong
   \begin{cases}
     GW_{3-q}^{3} (k) & \text{if } p = 0 \\
     GW_{3-n-q}^{3} (k) & \text{if } p = n-1 \\
     0 & \text{else.}
   \end{cases}
\end{equation*}
\end{center}

In particular, we have that $F_{p}/F_{p+1} \cong E(3)_{\infty}^{p,2-p} = 0$ if $0 < p \neq n-1$ and $F_{n-1}/F_{n} = GW_{0}^{3} (k)$. Hence $F_{n+1} = F_{n} = 0$ and $F_{1} = F_{2} = ... = F_{n-1}$. It follows from the exact sequence 

\begin{center}
$0 \rightarrow F_{n} \rightarrow F_{n-1} \rightarrow F_{n-1}/F_{n} \rightarrow 0$
\end{center}

that $GW_{1,red}^{3} (X) = F_{1} = GW_{0}^{3} (k)$. But $GW_{0}^{3} (k) \cong \mathbb{Z}/2\mathbb{Z}$ by \cite[Lemma 4.1]{FS1}. This proves the proposition.
\end{proof}

To conclude this section, we finally prove some cohomological criteria for the $2$-divisibility of the groups $W_E (R)$ and $W_{SL} (R)$ of any smooth affine algebra $R$ of dimension $4$ over an algebraically closed field $k$ with $char(k) \neq 2$. For the following proposition, recall that one can define the Milnor-Witt $K$-theory $K^{MW}_{\ast} (F)$ of a field $F$, which is a $\mathbb{Z}$-graded ring with explicit generators and relations given in \cite[Definition 2.1]{Mo}; for example, the group $K^{MW}_{0} (F)$ is canonically isomorphic to the Grothendieck-Witt ring $GW_{0}^{0} (F) = GW (F)$ of non-degenerate symmetric bilinear forms over $F$. We denote by $\textbf{K}^{MW}_{i}$ the associated Milnor-Witt $K$-theory sheaves in degree $i \in \mathbb{Z}$. For a general introduction to Milnor-Witt $K$-theory, we refer the reader to \cite[Section 2]{Mo}.

\begin{Prop}\label{P2.8}
Let $X=Spec(R)$ be a smooth affine fourfold over an algebraically closed field $k$ with $char (k) \neq 2$. Then $W_E (R)$ is $2$-divisible if $H^{2} (X,\textbf{\textsc{K}}_{3}^{MW})$ and $H^{3} (X, \textbf{\textsc{GW}}_{4}^{3})$ are $2$-divisible. Furthermore, $W_{SL} (R)$ is $2$-divisible if $CH^{3} (X) = CH^{4} (X) = 0$ and $H^{2} (X, \textbf{\textsc{I}}^{3})$ is $2$-divisible.
\end{Prop}

\begin{proof}
We use the Gersten-Grothendieck-Witt spectral sequence $E(3)$ associated to $X$. We have a filtration

\begin{center}
$0 = F_{5} \subset F_{4} \subset ... \subset GW_{1,red}^{3} (R) = F_{1} \subset GW_{1}^{3} (R) = F_{0}$
\end{center}

with $F_{p}/F_{p+1} \cong E(3)_{\infty}^{p,2-p}$ for all $p$. The terms $E(3)_{2}^{p,q}$ on the second page are $H^{p} (X, \textbf{GW}_{3-q}^{3})$ for $0 \leq p \leq 4$ and $p+q \leq 3$ and $0$ elsewhere.\\
First of all, \cite[Lemma 2.2]{FRS} implies that $E(3)_{1}^{p,1} = 0$ for all $p$. Therefore $E(3)_{\infty}^{1,1} = 0$ and hence $F_{2} = W_E (R)$. Moreover, since $k$ is algebraically closed, the limit term $F_{4} = E(3)_{\infty}^{4,-2}$ is a quotient of $\oplus_{x \in X^{(4)}} k^{\times}$ and therefore $2$-divisible. Altogether, we have two short exact sequences

\begin{center}
$0 \rightarrow F_{3} \rightarrow W_E (R) \rightarrow E(3)_{\infty}^{2,0} \rightarrow 0$,\\
$0 \rightarrow F_{4} \rightarrow F_{3} \rightarrow E(3)_{\infty}^{3,-1} \rightarrow 0$,
\end{center}

where $F_{4}$ is $2$-divisible. In particular, $W_E (R)$ is $2$-divisible as soon as $E(3)_{\infty}^{2,0}$ and $E(3)_{\infty}^{3,-1}$ are $2$-divisible.\\
However, $E(3)_{\infty}^{3,-1}$ is a quotient of $H^{3} (X, \textbf{GW}_{4}^{3})$. Furthermore, we know that $E(3)_{2}^{2,0}$ is precisely $H^{2} (X, \textbf{GW}_{3}^{3}) \cong H^{2} (X, \textbf{K}_{3}^{MW})$. Hence $E(3)_{\infty}^{2,0}$ is precisely the kernel of the differential mapping into $E(3)_{2}^{4,-1} \cong H^{4} (X, \textbf{GW}_{4}^{3})$. But by the fact that $CH^{4} (X)$ is $2$-divisible and by \cite[Proposition 3.6.4]{AF3}, we can conclude that $H^{4} (X, \textbf{GW}_{4}^{3}) = 0$. Thus, the limit term $E(3)_{\infty}^{2,0}$ is precisely $H^{2} (X, \textbf{K}_{3}^{MW})$ and the first statement follows.\\
For the second statement, we will use the Brown-Gersten-Quillen spectral sequence $E'(3)$ associated to $X$ which has terms of the form

\begin{center}
\begin{equation*}
   E'(3)_{1}^{p,q} \cong
   \begin{cases}
     \bigoplus_{x_{p} \in X^{(p)}} K_{3-p-q}^{Q} (k(x_{p})) & \text{if } 0 \leq p \leq 4 \text{ and } 3 \geq p+q\\
     0 & \text{else}
   \end{cases}
\end{equation*}
\end{center}

on the first page and converges to $K^{Q}_{3-\ast} (X)$. The group $SK_{1} (R)$ can be computed via the limit terms $E'(3)_{\infty}^{p,2-p}$: there is a filtration

\begin{center}
$0 = F'_{5} \subset F'_{4} \subset ... \subset SK_{1} (R) = F'_{1} \subset K_{1} (R) = F'_{0}$
\end{center}

with $F'_{p}/F'_{p+1} \cong E'(3)_{\infty}^{p,2-p}$ for all $p$. Moreover, the terms $E'(3)_{2}^{p,q}$ on the second page are isomorphic to $H^{p} (X, \textbf{K}^{Q}_{3-q})$ for $0 \leq p \leq 4$ and $p+q \leq 3$.\\
By construction of both the Brown-Gersten-Quillen and the Gersten-Grothendieck-Witt spectral sequences, the hyperbolic morphism induces a morphism of spectral sequences. Hence we get a commutative diagram

\begin{center}
$\begin{xy}
  \xymatrix{
      0 \ar[r] \ar[d] & F'_{3} \ar[d]^{H_{1,3}} \ar[r] & SK_{1} (X) \ar[d]^{H_{1,3}} \ar[r] & F'_{1}/F'_{3} \ar[d]^{H_{1,3}} \ar[r] & 0 \ar[d]\\
      0 \ar[r] & F_{3} \ar[r] & W_E (X) \ar[r] & H^{2} (X, \textbf{K}_{3}^{MW}) \ar[r] & 0   
  }
\end{xy}$
\end{center}

with exact rows. If $H^{3} (X, \textbf{GW}_{4}^{3})$ is $2$-divisible (in particular, if $CH^{3} (X)$ is $2$-divisible by Proposition \ref{P2.6}), then we have already seen above that $F_{3}$ is $2$-divisible. Since $k$ is algebraically closed, $W_{SL} (R)$ is a $2$-torsion group. Hence the snake lemma induces an isomorphism $W_{SL} (R) \xrightarrow{\cong} H^{2} (X, \textbf{K}_{3}^{MW})/H_{1,3}(F'_{1}/F'_{3})$. In particular, there is a surjection $H^{2} (X, \textbf{K}_{3}^{MW})/H_{1,3}(F'_{2}/F'_{3}) \rightarrow W_{SL} (R)$.\\
Since $CH^{4} (X) = H^{4} (X, \textbf{K}^{Q}_{4}) = 0$, the differential starting at ${E'(3)}^{2,0}_{2}$ maps into a trivial group and hence its kernel is isomorphic to $H^{2} (X, \textbf{K}_{3}^{Q})$. It then follows that the group $F'_{2}/F'_{3} \cong E'(3)_{\infty}^{2,0}$ will be a quotient of ${E'(3)}^{2,0}_{2} = H^{2} (X, \textbf{K}_{3}^{Q})$ and hence $H^{2} (X, \textbf{K}_{3}^{MW})/H_{1,3}(F'_{2}/F'_{3}) \cong H^{2} (X, \textbf{K}_{3}^{MW})/H_{3,3}(H^{2} (X, \textbf{K}_{3}^{Q}))$. Finally, as the homomorphism $H^{2} (X, \textbf{K}_{3}^{Q}) \rightarrow H^{2} (X, 2\textbf{K}_{3}^{Q})$ is surjective, the long exact sequence of cohomology groups associated to the short exact sequence 

\begin{center}
$0 \rightarrow 2\textbf{K}_{3}^{M} \rightarrow \textbf{K}_{3}^{MW} \rightarrow \textbf{I}^{3} \rightarrow 0$
\end{center}

(whose existence is due to Morel and follows from Voevodsky's resolution of the Bloch-Kato conjectures) shows that $H^{2} (X, \textbf{K}_{3}^{MW})/H_{3,3} (X, \textbf{K}_{3}^{Q}) \cong H^{2} (X, \textbf{I}^{3})$. This yields the second statement.
\end{proof}

\section{The generalized Vaserstein symbol modulo $SL$}\label{The generalized Vaserstein symbol modulo $SL$}\label{3}

In this section, we prove that the generalized Vaserstein symbol defined in \cite{Sy} for any projective module $P_{0}$ of rank $2$ over a commutative ring $R$ together with a fixed trivialization $\theta_{0}: R \xrightarrow{\cong} \det (P_{0})$ of its determinant descends to a well-defined map $V_{\theta_{0}}: Um (P_{0} \oplus R)/SL (P_{0} \oplus R) \rightarrow \tilde{V}_{SL} (R)$, which we will call the generalized Vaserstein symbol modulo $SL$. We will study this map and give criteria for its surjectivity and injectivity for Noetherian rings of dimension $\leq 4$. As an application of this, we will be able to give a criterion for the triviality of the orbit space $Um (P_{0} \oplus R)/SL (P_{0} \oplus R)$ for such rings. Motivated by this criterion, we study symplectic orbits of unimodular rows and prove in particular that $Sp_{d} (R)$ acts transitively on $\mathit{Um}_{d} (R)$ whenever $d$ is divisible by $4$ and $R$ is a smooth affine algebra of dimension $d$ over an algebraically closed field $k$ with $d! \in k^{\times}$. As an immediate consequence of this, we will prove that $\mathit{Um}_{3} (R)/SL_{3} (R)$ is trivial if and only if $\tilde{V}_{SL} (R)$ is trivial whenever $R$ is a smooth affine algebra of dimension $4$ over an algebraically closed field $k$ with $6 \in k^{\times}$. Finally, we can also give cohomological criteria for the triviality of $\mathit{Um}_{3} (R)/SL_{3} (R)$ in this situation.

\subsection{The generalized Vaserstein symbol}\label{The generalized Vaserstein symbol}\label{3.1}

Let $R$ be a commutative ring and $P_0$ be a projective $R$-module of rank $2$. We assume that $P_0$ admits a trivialization $\theta_{0}: R \rightarrow \det(P_0)$ of its determinant.\\
Let us recall the definition of the generalized Vaserstein symbol associated to $\theta_{0}$: We denote by $\chi_0$ the canonical non-degenerate alternating form on $P_0$ given by $P_0 \times P_0 \rightarrow R, (p,q) \mapsto \theta_{0}^{-1} (p \wedge q)$.\\
Now let $Um (P_0 \oplus R)$ be the set of epimorphism $P_0 \oplus R \rightarrow R$. Any element $a$ of $Um (P_0 \oplus R)$ gives rise to an exact sequence of the form

\begin{center}
$0 \rightarrow P(a) \rightarrow P_0 \oplus R \xrightarrow{a} R \rightarrow 0$,
\end{center}

\noindent where $P(a) = \ker (a)$. Any section $s: R \rightarrow P_0 \oplus R$ of $a$ determines a canonical retraction $r_{s}: P_0 \oplus R \rightarrow P(a)$ given by $r_{s}(p)= p - s a(p)$ and an isomorphism $i_{s}: P_0 \oplus R \rightarrow P(a) \oplus R$ given by $i_{s}(p)  = a(p) + r_{s}(p)$.\\
The exact sequence above yields an isomorphism $\det(P_0) \cong \det (P(a))$ (independent of $s$) and therefore an isomorphism $\theta : R \rightarrow \det (P(a))$ obtained by composing with $\theta_0$. We denote by $\chi_a$ the non-degenerate alternating form on $P(a)$ given by $P(a) \times P(a) \rightarrow R, (p,q) \mapsto \theta^{-1} (p \wedge q)$.\\
Altogether, we obtain a non-degenerate alternating form

\begin{center}
$V (a,s) = {(i_{s} \oplus 1)}^{t} (\chi_a \perp \psi_2) {(i_{s} \oplus 1)}$
\end{center}

on $P_{0} \oplus R^{2}$, which depends on the section $s$ of $a$. Nevertheless, assigning to $a \in Um (P_{0} \oplus R)$ the element

\begin{center}
$V_{\theta_{0}} (a) = [P_0 \oplus R^2, \chi_0 \perp \psi_2, {(i_{s} \oplus 1)}^{t} (\chi_a \perp \psi_2) {(i_{s} \oplus 1)}]$
\end{center}

induces a well-defined map

\begin{center}
$V_{\theta_{0}}: Um (P_{0} \oplus R)/E (P_{0} \oplus R) \rightarrow \tilde{V} (R)$
\end{center}

called the generalized Vaserstein symbol associated to $\theta_{0}$ (cp. \cite[Theorem 1]{Sy}). If there is no ambiguity, we denote $V_{\theta_{0}}$ simply by $V$.\\
Of course, if $P_{0} = R^{2}$, then we have a canonical trivialization $\theta_{0}$ of $\det(R^{2})$ given by $1 \mapsto e_{1} \wedge e_{2}$, where $e_{1} = (1,0), e_{2} = (0,1) \in R^{2}$. The generalized Vaserstein symbol associated to $-\theta_{0}$ is just the classical one introduced in \cite[\S 5]{SV}.\\
Now let us return to the general case of a projective $R$-module $P_{0}$ of rank $2$ with a fixed trivialization $\theta_{0}$. We compose the generalized Vaserstein symbol $V = V_{\theta}$ with the canonical epimorphism $\tilde{V} (R) \rightarrow \tilde{V}_{SL} (R)$:

\begin{Thm}\label{T3.1}
Let $\varphi \in SL (P_{0} \oplus R)$ and $a \in Um (P_{0} \oplus R)$. Then there is an equality $V (a) = V (a \varphi)$ in $\tilde{V}_{SL} (R)$. In particular, we obtain a well-defined map $V: Um (P_0 \oplus R)/SL (P_0 \oplus R) \rightarrow \tilde{V}_{SL} (R)$, which we call the generalized Vaserstein symbol modulo $SL$.
\end{Thm}
 
\begin{proof}
First of all, let $\varphi \in SL (P_0 \oplus R)$ and let $s: R \rightarrow P_0 \oplus R$ be a section of $a \in Um (P_0 \oplus R)$. Then ${\varphi}^{-1} s$ is a section of $a \varphi$. We let $i: P_0 \oplus R \rightarrow P(a) \oplus R$ and $j: P_0 \oplus R \rightarrow P(a \varphi) \oplus R$ be the isomorphisms induced by the sections $s$ and ${\varphi}^{-1} s$. Obviously, it suffices to show that

\begin{center}
${(\varphi \oplus 1)}^{t} {(i \oplus 1)}^{t} (\chi_a \perp \psi_2) {(i \oplus 1)} {(\varphi \oplus 1)} = {(j \oplus 1)}^{t} (\chi_{(a \varphi)} \perp \psi_2) {(j \oplus 1)}$.
\end{center}

One can check easily that $(i \oplus 1) (\varphi \oplus 1) = ((\varphi \oplus 1) \oplus 1) (j \oplus 1)$, where by abuse of notation we understand $\varphi$ as the induced isomorphism $P (a \varphi) \rightarrow P(a)$. Hence it suffices to show that ${\varphi}^{t} \chi_a \varphi = \chi_{a \varphi}$.\\
For this, we let $(p,q)$ a pair of elements in $P (a \varphi)$; by definition, $\chi_{a \varphi}$ sends these elements to the image of $p \wedge q$ under the isomorphism $\det (P (a \varphi)) \cong R$. This element can also be described as the image of $p \wedge q \wedge {\varphi}^{-1} s(1)$ under the isomorphism $\det (P_0 \oplus R) \cong R$.\\
Analogously, the alternating form ${\varphi}^{t} \chi_a \varphi$ sends $(p,q)$ to the image of the element ${\varphi} (p) \wedge {\varphi} (q) \wedge s(1)$ under the isomorphism $\det (P_0 \oplus R) \cong R$. Since $\varphi$ has determinant $1$, the automorphism of $\det (P_{0} \oplus R)$ induced by $\varphi$ is the identity (cp. \cite[Lemma 2.11]{Sy}). This proves the desired equality ${\varphi}^{t} \chi_a \varphi = \chi_{a \varphi}$.
\end{proof}

\subsection{An exact sequence}\label{An exact sequence}\label{3.2}

In this section, we assume that $R$ is Noetherian commutative ring of Krull dimension $\leq 4$. Let us fix some notation: We let $P_0$ be a projective $R$-module of rank $2$. For all $n \geq 3$, let $P_n = P_0 \oplus R e_3 \oplus ... \oplus R e_n$ be the direct sum of $P_0$ and free $R$-modules $R e_i$, $3 \leq i \leq n$, of rank $1$ with explicit generators $e_i$. We will sometimes omit these explicit generators in the notation. We denote by $\pi_{k, n}: P_n \rightarrow R$ the projections onto the free direct summands of rank $1$ with index $k = 3, ...,n$.\\
We assume that $P_0$ admits a trivialization $\theta_{0}: R \rightarrow \det(P_0)$ of its determinant. By abuse of notation, we denote by $V = V_{\theta}: Um (P_{0} \oplus R)/E (P_{0} \oplus R) \rightarrow \tilde{V}_{SL} (R)$ the composite of the generalized Vaserstein symbol associated to $\theta$ and the canonical epimorphism $\tilde{V} (R) \rightarrow \tilde{V}_{SL} (R)$.

\begin{Thm}\label{T3.2}
Assume that $SL (P_{5})$ acts transitively on $Um (P_{5})$. Then the map $V: Um (P_{0} \oplus R)/E (P_{0} \oplus R) \rightarrow \tilde{V}_{SL} (R)$ is surjective.
\end{Thm}

\begin{proof}
Let $\beta \in \tilde{V}_{SL} (R)$. Since $\dim (R) \leq 4$, we know that $Um (P_{n}) = \pi_{n,n} E (P_{n})$ for all $n \geq 6$. Therefore every element in $\tilde{V} (R)$ is of the form $[P_{6}, \chi_{0} \perp \psi_{4}, \chi]$ for some non-degenerate alternating form $\chi$ on $P_{6}$ (cp. \cite[Lemma 2.10]{Sy}); hence the same holds for any element in $\tilde{V}_{SL} (R)$. Consequently, we can write $\beta = [P_{6}, \chi_{0} \perp \psi_{4}, \chi]$.\\
Now let $d = \chi(-, e_{6}): P_{5} \rightarrow R$. Since $d$ can be locally checked to be an epimorphism, there is an automorphism ${\varphi} \in SL (P_{5})$ such that ${d} {\varphi} = \pi_{5, 5}$. Then the alternating form ${\chi'} = {(\varphi \oplus 1)}^{t} \chi  ({\varphi} \oplus 1)$ satisfies that ${\chi'} (-, e_{6}) : P_{5} \rightarrow R$ is just $\pi_{5, 5}$. Now we simply define $c = {\chi'} (-, e_{5}): P_{5} \rightarrow R$ and let ${\varphi}_{c} = id_{P_{6}} + c e_{6}$ be the elementary automorphism on $P_{6}$ induced by $c$; then ${{\varphi}_{c}}^{t} {\chi'} {\varphi}_{c} = \psi \perp \psi_{2}$ for some non-degenerate alternating form $\psi$ on $P_{4}$. Since all the isometries we used have determinant $1$, we conclude that $\beta = [P_{4}, \chi_{0} \perp \psi_{2}, \psi]$. As any element of this form lies in the image of the generalized Vaserstein symbol (cp. \cite[Lemma 4.4]{Sy}), this proves the theorem.
\end{proof}

We remark that the assumption in the last theorem is satisfied if $R$ is an algebra of dimension $\leq 4$ over an infinite perfect field $k$ of cohomological dimension $\leq 1$ with $6 \in k^{\times}$ (cp. \cite{S1}, \cite{S4} and \cite{B}) or if $R$ is a Noetherian ring of dimension $\leq 3$ (\cite[Chapter IV, Corollary 3.5]{HB}).\\
Now let us study the fibers of the map $V: Um (P_{0} \oplus R)/E (P_{0} \oplus R) \rightarrow \tilde{V}_{SL} (R)$. For this, we will now describe an action of $SL (P_{4})$ on $Um (P_{0} \oplus R)/E (P_{0} \oplus R)$.\\
First of all, note that $E (P_{4})$ is a normal subgroup of $SL (P_{4})$: for if we let $\varphi \in SL (P_{4})$ and $\varphi' \in E (P_{4})$, then there is a natural isotopy (cp. \cite[Definition before Theorem 2.14]{Sy}) from $id_{P_{4}}$ to ${\varphi}^{-1} \varphi' \varphi$. By \cite[Theorem 2.12]{Sy} and Suslin's normality theorem (cp. \cite{S3}), it follows that ${\varphi}^{-1} \varphi' \varphi \in E (P_{4})$.\\
Now let $\varphi \in SL (P_{4})$ and $a \in Um (P_{0} \oplus R)$. We choose a section $s: R \rightarrow P_{0} \oplus R$ of $a$ and obtain a non-degenerate alternating form

\begin{center}
$V (a,s) = {(i_{s} \oplus 1)}^{t} (\chi_a \perp \psi_{2}) (i_{s} \oplus 1)$
\end{center}

as in the definition of the generalized Vaserstein symbol. Then we consider the alternating form $\varphi^{t} V (a,s) \varphi$. By abuse of notation, we also denote by $a$ the class of $a$ in $Um (P_{0} \oplus R)/E (P_{0} \oplus R)$ and define $a \cdot \varphi$ to be the class in $Um (P_{0} \oplus R)/E (P_{0} \oplus R)$ represented by $\varphi^{t} V (a,s) \varphi (-, e_{4}): P_{0} \oplus R \rightarrow R$.\\
Now let us show that this assignment gives a well-defined right action of $SL (P_{4})$ on $Um (P_{0} \oplus R)/E (P_{0} \oplus R)$: If we choose another section $s'$ of $a$, then there is $\varphi' \in E (P_{4})$ such that $\varphi' V (a,s') \varphi' = V (a,s)$ (cp. the proof of \cite[Theorem 4.1]{Sy}). Since $E (P_{4})$ is a normal subgroup of $SL (P_{4})$, it follows that

\begin{center}
${(\varphi)}^{t} V (a,s) \varphi = {(\varphi'')}^{t} {(\varphi)}^{t} V (a,s') \varphi \varphi''$ 
\end{center}

for some $\varphi'' \in E (P_{4})$. The lemma below will hence imply that our assignment does not depend on the choice of the section $s$ of $a$.\\
Similarly, if $a' = a \varphi'$ for $\varphi' \in E (P_{0} \oplus R)$, then $V (a',s') = {(\varphi' \oplus 1)}^{t} V (a,s) (\varphi' \oplus 1)$, where $s' = {(\varphi')}^{-1} s$ (this follows from the proof of \cite[Theorem 4.3]{Sy}). Again, since $E (P_{4})$ is normal in $SL (P_{4})$, it follows that

\begin{center}
${(\varphi)}^{t} V (a,s) \varphi = {(\varphi'')}^{t} {(\varphi)}^{t} V (a',s') \varphi \varphi''$
\end{center}

for some $\varphi'' \in E (P_{4})$. The following lemma then also implies that our assignment does only depend on the class of $a$ in $Um (P_{0} \oplus R)/E (P_{0} \oplus R)$.

\begin{Lem}
Let $\chi, \chi'$ be non-degenerate alternating forms on $P_{4}$ and, moreover, let $a = \chi (-,e_{4}), a' = \chi' (-,e_{4}) \in Um (P_{0} \oplus R)$. If ${\varphi}^{t} \chi \varphi = \chi'$ for some $\varphi \in E (P_{4})$, then the classes of $a$ and $a'$ coincide in $Um (P_{0} \oplus R)/E (P_{0} \oplus R)$.
\end{Lem}

\begin{proof}
First of all, the group $E (P_{4})$ is generated by elementary automorphisms $\varphi_{g} = id_{P_{4}} + g$, where $g$ is a homomorphism

\begin{itemize}
\item[1)] $g: R e_{3} \rightarrow P_{0}$,
\item[2)] $g: P_{0} \rightarrow R e_{3}$,
\item[3)] $g: R e_{3} \rightarrow R e_{4}$ or
\item[4)] $g: R e_{4} \rightarrow R e_{3}$.
\end{itemize}

Furthermore, we can write $\chi = V (a,s)$ and $\chi' = V (a', s')$ for sections $s$ and $s'$ of $a$ and $a'$ respectively (cp. the proof of \cite[Lemma 4.4]{Sy}). Hence it suffices to show the following: If $\varphi_{g}^{t} V (a,s) \varphi_{g} = V (a',s')$ for some $g$ as above, then $a' = a \psi$ for some $\psi \in E (P_{0} \oplus R)$. The only non-trivial case is the last one, i.e. if $g$ is a homomorphism $R e_{4} \rightarrow R e_{3}$.\\
For this, we let $g: R e_{4} \rightarrow R e_{3}$ and let $\varphi_{g}$ be the induced elementary automorphism of $P_{4}$ and we assume that

\begin{center}
$\varphi_{g}^{t} V (a,s) \varphi_{g} = V (a', s')$
\end{center}

for some epimorphism $a': P_{0} \oplus R e_{3} \rightarrow R$ with section $s'$. We then write $a$ as $a = (a_{0}, a_{R})$, where $a_{0}$ is the restriction of $a$ to $P_{0}$ and $a_{R} = a (e_{3})$. Moreover, we define $p = \pi_{P_{0}} (s (1))$. From now on, we interpret the alternating form $\chi_{0}$ in the definition of the generalized Vaserstein symbol as an alternating isomorphism $\chi_{0}: P \rightarrow P^{\vee}$. One can verify locally that

\begin{center}
$a' = (a_{0} - g(1) \cdot \chi_{0} (p), a_{R})$.
\end{center}

Then let us define an automorphism $\psi$ of $P_{3}$ as follows: We first define an endomorphism of $P_{0}$ by 

\begin{center}
$\psi_{0} = id_{P_{0}} - g(1) \cdot \pi_{P_{0}} \circ s \circ \chi_{0} (p) : P_{0} \rightarrow P_{0}$
\end{center}
 
and we also define a morphism $P_{0} \rightarrow R e_{3}$ by
 
\begin{center}
$\psi_{R} = - g(1) \cdot \pi_{R} \circ s \circ \chi_{0} (p): P_{0} \rightarrow R$.
\end{center}

Then we consider the endomorphism of $P_{0} \oplus R$ given by
 
\begin{center}
$\psi =
\begin{pmatrix}
\psi_{0} & 0 \\
\psi_{R} & id_R
\end{pmatrix}
$.
\end{center}

First of all, this endomorphism coincides up to an elementary automorphism with

\begin{center}
$
\begin{pmatrix}
\psi_{0} & 0 \\
0 & id_R
\end{pmatrix}
$.
\end{center}

Since $\chi_{0} (p) \circ \pi_{P_{0}} \circ s = 0$, this endomorphism is an element of $E (P_{0} \oplus R)$ by \cite[Lemma 2.6]{Sy}. Hence the same holds for $\psi$. Finally, one can check easily that $a \psi = a'$ by construction.
\end{proof}

As indicated above, the previous lemma shows that our previous assignment gives a well-defined map

\begin{center}
$Um (P_{0} \oplus R)/E (P_{0} \oplus R) \times SL (P_{4}) \xrightarrow{-\cdot-} Um (P_{0} \oplus R)/E (P_{0} \oplus R)$.
\end{center}

Note that if $a \in Um (P_{0} \oplus R)$ with section $s$ and $\varphi \in SL (P_{4})$, then it follows from the proof of \cite[Lemma 4.4]{Sy} that the alternating form ${\varphi}^{t} V (a,s) \varphi$ equals $V (a \cdot \varphi, s')$ for some section $s'$ of $a \cdot \varphi$. It follows that the map above is indeed a right action of $SL (P_{4})$ on $Um (P_{0} \oplus R)/E (P_{0} \oplus R)$. In fact, the previous lemma shows that this action descends to a well-defined action of $SL (P_{4})/E (P_{4})$ on $Um (P_{0} \oplus R)/E (P_{0} \oplus R)$.

\begin{Lem} \label{L3.4}
Let $\chi_{1}$ and $\chi_{2}$ be non-degenerate alternating forms on $P_{2n}$ such that ${\varphi}^{t} (\chi_{1} \perp \psi_{2}) \varphi = \chi_{2} \perp \psi_{2}$ for some $\varphi \in SL (P_{2n+2})$. Furthermore, let $\chi = \chi_{1} \perp \psi_{2}$. If $SL (P_{2n+2}) e_{2n+2} = {Sp} (\chi) e_{2n+2}$ holds, then one has ${\psi}^{t} \chi_{2} \psi = \chi_{1}$ for some $\psi \in SL (P_{2n})$.
\end{Lem}

\begin{proof}
Let ${\psi''} e_{2n+2} = \varphi e_{2n+2}$ for some ${\psi''} \in {Sp} (\chi)$. Then we set ${\psi'} = {(\psi'')}^{-1} \varphi$. Since ${(\psi')}^{t} (\chi_{1} \perp \psi_{2}) {\psi'} = \chi_{2} \perp \psi_{2}$, the composite $\psi: P_{2n} \xrightarrow{\psi'} P_{2n+2} \rightarrow P_{2n}$ and $\psi'$ satisfy the following conditions:

\begin{itemize}
\item ${\psi'} (e_{2n+2}) = e_{2n+2}$;
\item $\pi_{2n+1, 2n+2} {\psi'} = \pi_{2n+1, 2n+2}$;
\item ${\psi}^{t} \chi_{1} \psi = \chi_{2}$.
\end{itemize}
 
These conditions imply that $\psi$ equals ${\psi'}$ up to elementary morphisms of $P_{2n+2}$ and hence has determinant $1$ as well. This finishes the proof.
\end{proof}

\begin{Thm}\label{T3.5}
Let $a,a' \in Um (P_{0} \oplus R)$. Then $V (a) = V (a')$ in $\tilde{V}_{SL} (R)$ if and only if $a \cdot \varphi = a'$ in $Um (P_{0} \oplus R)/E (P_{0} \oplus R)$ for some $\varphi \in SL (P_{4})$.
\end{Thm}

\begin{proof}
We let $s,s': R \rightarrow P_{0} \oplus R$ be sections of $a$ and $a'$ and $V (a,s)$ and $V (a',s')$ be the alternating forms induced by $s$ and $s'$, which appear in the definition of the Vaserstein symbol. Now assume that $V (a) = V (a')$. Since by assumption $\dim (R) \leq 4$, we know that $E (P_{n}) e_{n} = Um (P_{n})$ for all $n \geq 6$. In particular, one has $(E (P_{2n}) \cap Sp (\chi)) e_{2n} = Um (P_{2n})$ for all $n \geq 3$ and all non-degenerate alternating forms on $P_{2n}$ (cp. \cite[Lemma 2.8]{Sy}). Hence we can apply Lemma \ref{L2.1} and Lemma \ref{L3.4} in order to deduce that ${\varphi}^{t} V (a,s) \varphi = V (a',s')$ for some $\varphi \in SL (P_{4})$. By definition of the action of $SL (P_{4})$ on $Um (P_{0} \oplus R)/E (P_{0} \oplus R)$, this means that $a \cdot \varphi = a'$.\\
Conversely, assume that $a \cdot \varphi = a'$ for some $\varphi \in SL (P_{4})$. By definition, this means that ${\varphi}^{t} V (a,s) \varphi = V (a'',s'')$, where the class of $a'' \in Um (P_{0} \oplus R)$ coincides with the class of $a'$ in $Um (P_{0} \oplus R)/E (P_{0} \oplus R)$ and $s''$ is a section of $a''$. In particular, it follows from the proofs of \cite[Theorems 4.1 and 4.3]{Sy} that there exists $\psi \in E (P_{4})$ such that ${\psi}^{t} {\varphi}^{t} V (a,s) \varphi \psi = V (a', s')$. This clearly implies that $V (a) = V (a')$ in $\tilde{V}_{SL} (R)$.
\end{proof}

For any Noetherian ring $R$ of dimension $\leq 4$, we have established the following exact sequence of groups and pointed sets whenever $SL (P_{5})$ acts transitively on $Um (P_{5})$:

\begin{center}
$SL (P_{4}) \Rightarrow Um (P_{0} \oplus R)/E (P_{0} \oplus R) \xrightarrow{V} \tilde{V}_{SL} (R) \rightarrow 0$.
\end{center}

In this situation, we mean by exactness at $Um (P_{0} \oplus R)/E (P_{0} \oplus R)$ that two classes in $Um (P_{0} \oplus R)/E (P_{0} \oplus R)$ represented by $a, a' \in Um (P_{0} \oplus R)$ satisfy $V (a) = V (a')$ in $\tilde{V}_{SL} (R)$ if and only if $a \varphi = a'$ for some $\varphi \in SL (P_{4})$.\\
Furthermore, there is a well-defined right action of $SK_{1} (R)$ on $W_E (R) \cong \tilde{V}_{SL} (R)$ given by the following assignment: If $\varphi \in SL_{2n} (R)$ and $\theta \in A_{2n} (R)$ represent elements of $SK_{1} (R)$ and $W_E (R)$, then $\theta \cdot \varphi$ is represented by the class of ${\varphi}^{t} \theta \varphi$ in $W_{E} (R)$. This action is compatible with the right action introduced above: Following \cite[Chapter III, Lemma 1.6]{We}, any finitely generated projective $R$-module $Q$ such that $P_{0} \oplus Q \cong R^{n}$ for some $n>0$ induces a well-defined group homomorphism $SL (P_{4}) \rightarrow SL_{n+2} (R)$. This induces a well-defined map $SL (P_{4}) \rightarrow SK_{1} (R)$ independent of the choice of $Q$. In fact, the map descends to a well-defined group homomorphism $St: SL (P_{4})/E (P_{4}) \rightarrow SK_{1} (R)$. One can then check easily that the diagram

\begin{center}
$\begin{xy}
  \xymatrix{
      Um (P_{3})/E (P_{3}) \times SL (P_{4})/E (P_{4}) \ar[r] \ar[d]^{V \times St}    &   Um (P_{3})/E (P_{3}) \ar[d]^{V} \ar[r] & Um (P_{3})/SL (P_{3}) \ar[d]^{V} \\
      \tilde{V} (R) \times SK_{1} (R) \ar[r]             &   \tilde{V} (R) \ar[r] & \tilde{V}_{SL} (R)   
  }
\end{xy}$
\end{center}

is commutative.\\
As a consequence of the previous theorem, we obtain the following criterion for the injectivity of the map $V: Um (P_{0} \oplus R)/SL (P_{0} \oplus R) \rightarrow \tilde{V}_{SL} (R)$:

\begin{Thm}\label{SLInjective}
The map $V: Um (P_{0} \oplus R)/SL (P_{0} \oplus R) \rightarrow \tilde{V}_{SL} (R)$ is injective if and only $SL (P_{4}) e_{4} = {Sp} (\chi) e_{4}$ for all non-degenerate alternating forms $\chi$ on $P_{4}$ such that $[P_{4}, \chi_{0}\perp\psi_{2},\chi] \in \tilde{V} (R)$.
\end{Thm}

\begin{proof}
First of all, assume that $SL (P_{4}) e_{4} = {Sp} (\chi) e_{4}$ for all non-degenerate alternating forms $\chi$ on $P_{4}$ such that $[P_{4}, \chi_{0}\perp\psi_{2},\chi] \in \tilde{V} (R)$. Now let $a,a' \in Um (P_{0} \oplus R)$ such that $V (a) = V (a')$. Then ${\varphi}^{t} V (a,s) \varphi = V (a',s')$ for some $\varphi \in SL (P_{4})$ and sections $s,s'$ of $a$ and $a'$ by the previous theorem. By assumption there is $\varphi' \in {Sp} (V (a,s))$ with $\varphi e_{4} = \varphi' e_{4}$. If we let $\varphi'' = {\varphi'}^{-1} \varphi$, then $\varphi'' e_{4} = e_{4}$ and ${\varphi''}^{t} V (a,s) \varphi'' = V (a',s')$. Thus, if we write

\begin{center}
$\varphi'' =
\begin{pmatrix}
\varphi''_{0} & 0 \\
\varphi''_{R} & 1
\end{pmatrix}
\in Aut (P_{3} \oplus R)$,
\end{center}

then $\varphi''_{0}$ has determinant $1$ and satisfies $a' = a \varphi''_{0}$. In particular, the classes of $a$ and $a'$ in $Um (P_{0} \oplus R)/SL (P_{0} \oplus R)$ coincide and $V$ is injective.\\
Conversely, assume that $V$ is injective. Let $\chi$ be an arbitrary non-degenerate alternating form on $P_{4}$ such that $[P_{4}, \chi_{0}\perp\psi_{2},\chi] \in \tilde{V} (R)$ and also let $\varphi \in SL (P_{4})$. We write $\chi = V (a,s)$ and ${\varphi}^{t} \chi \varphi = V (a',s')$ for $a,a' \in Um (P_{0} \oplus R)$ with sections $s$ and $s'$. Then obviously $V (a) = V (a')$. By assumption, there is $\varphi' \in SL (P_{0} \oplus R)$ with $a' = a \varphi'$ and hence ${(\varphi' \oplus 1)}^{t} V (a,s) (\varphi' \oplus 1) = V (a', s'')$, where $s''$ is a section of $a'$. Furthermore, there exists $\varphi'' \in E (P_{4})$ with $\varphi'' e_{4} = e_{4}$ such that ${\varphi''}^{t} V (a',s'') \varphi'' = V (a', s')$ (cp. the proof of \cite[Theorem 4.1]{Sy}). The automorphism $\beta = \varphi {\varphi''}^{-1} {(\varphi' \oplus 1)}^{-1}$ lies in ${Sp} {\chi}$ and satisfies $\beta e_{4} = \varphi e_{4}$, which proves the theorem.
\end{proof}

The proof of Theorem \ref{SLInjective} shows in particular the following statement:

\begin{Kor}\label{C3.7}
Assume that $SL (P_{5})$ acts transitively on $Um (P_{5})$. Then the orbit space $Um (P_{0} \oplus R)/SL (P_{0} \oplus R)$ is trivial if and only if $W_{SL} (R)$ is trivial and $SL (P_{4}) e_{4} = {Sp} (\chi_{0} \perp \psi_2) e_{4}$.
\end{Kor}


Recall that one of the basic tools to study the groups $W_E (R)$ and $W_{SL} (R)$ is the Karoubi periodicity sequence

\begin{center}
$K_{1}{Sp} (R) \rightarrow SK_{1} (R) \rightarrow W_E (R) \rightarrow K_{0}{Sp} (R) \rightarrow K_{0} (R)$.
\end{center}

Now let us simply denote by $W_E (R)/SL_{3} (R)$ the cokernel of the composite $SL_{3} (R) \rightarrow SK_{1} (R) \rightarrow W_E (R)$. Then we can deduce the following result from the previous corollary:

\begin{Kor}
Assume that $R$ is a smooth $4$-dimensional algebra over the algebraic closure $k = \bar{\mathbb{F}}_{q}$ of a finite field such that $6 \in k^{\times}$. Then the orbit space $\mathit{Um}_{3} (R)/SL_{3} (R)$ is trivial if and only if $W_E (R)/SL_{3} (R)$ is trivial.
\end{Kor}

\begin{proof}
As a matter of fact, it was proven in \cite[Corollary 7.8]{FRS} that the homomorphism $SL_{4} (R)/E_{4} (R) \xrightarrow{\cong} SK_{1} (R)$ is an isomorphism.\\
Now assume that $\mathit{Um}_{3} (R)/SL_{3} (R)$ is trivial. By Corollary \ref{C3.7}, this means that the map $SK_{1} (R) \rightarrow W_E (R)$ is surjective and ${Sp}_{4} (R) e_{4} = SL_{4} (R) e_{4}$. The second condition and the isomorphism $SL_{4} (R)/E_{4} (R) \cong SK_{1} (R)$ easily imply that any matrix in $SL_{4} (R)$ lies in $SL_{3} (R)$ up to a matrix in ${Sp}_{4} (R) E_{4} (R)$. Since elements in ${Sp}_{4} (R) E_{4} (R)$ are sent to $0$ in $W_E (R)$ under the map $SK_{1} (R) \rightarrow W_E (R)$, this immediately implies that $W_E (R)/SL_{3} (R) = W_{SL} (R) = 0$.\\
Conversely, assume that $W_E (R)/SL_{3} (R)$ is trivial. Then $W_{SL} (R)$ is obviously trivial. Now let $\varphi \in SL_{4} (R)$. Then the class of the matrix ${\varphi}^{t} \psi_{4} \varphi$ is trivial in $W_E (R)/SL_{3} (R)$. By the Karoubi periodicity sequence, this means that there exists a matrix $\varphi' \in SL_{3} (R)$ such that $\varphi {(\varphi' \oplus 1)}^{-1}$ is in the image of the map $K_{1}{Sp} (R) \rightarrow SK_{1} (R)$. Since $\dim (R) = 4$, $K_{1}{Sp} (R)$ is generated by $Sp_{4} (R)$; the isomorphism $SL_{4} (R)/E_{4} (R) \cong SK_{1} (R)$ then implies that ${\varphi''}^{-1} \varphi {(\varphi' \oplus 1)}^{-1}$ lies in $E_{4} (R)$ for some $\varphi'' \in Sp_{4} (R)$. Since for any $v \in \mathit{Um}^{t}_{4} (R)$ one has $E_{4} (R) v = (E_{4} (R) \cap {Sp}_{4} (R)) v$, it follows that there is an element $\psi \in E_{4} (R) \cap {Sp}_{4} (R)$ with ${\varphi''}^{-1} \varphi {(\varphi' \oplus 1)}^{-1} e_{4} = \psi e_{4}$. Since $\varphi {(\varphi' \oplus 1)}^{-1} e_{4} = \varphi e_{4}$, it follows that $\varphi e_{4} = \varphi'' \psi e_{4}$ and $\varphi'' \psi \in {Sp}_{4} (R)$. This proves the corollary.
\end{proof}

\begin{Kor}
Assume that $R$ is a smooth affine algebra of dimension $3$ over an algebraically closed field $k$ with $char(k) \neq 2$. Then we have an equality $Sp (\chi_{0}\perp\psi_{2}) e_{4} = Unim.El. (P_{4})$.
\end{Kor}

\begin{proof}
Let $X = Spec(R)$. Using the usual Postnikov tower techniques in motivic homotopy theory, it was proven in \cite[Theorem 6.6]{AF2} that the pointed map $\mathcal{V}_{2} (R) \rightarrow CH^{1} (X) \times CH^{2} (X)$ induced by the first and second Chern classes is a bijection. If one applies the same methods to $BSL_{2}$ instead of $BGL_{2}$, one can analogously deduce a bijection $\mathcal{V}_{2}^{o} (R) \xrightarrow{\cong} CH^{2} (X)$. Hence the isomorphism class of an oriented vector bundle is uniquely determined by its second Chern class. But stably isomorphic oriented vector bundles have the same total Chern class and must therefore be isomorphic. Hence it follows from Section \ref{2.1} that $Um (P_{0} \oplus R)/SL (P_{0} \oplus R)$ is trivial. Since $SL (P_{4}) e_{4} = Unim.El. (P_{4})$ and $SL (P_{5})$ acts transitively on $Um (P_{5})$, the result follows by Corollary \ref{C3.7}.
\end{proof}

Now let us briefly recall some results on the Bass-Quillen conjecture. Let us consider the following general statement for a commutative ring $R$:

\begin{itemize}
\item[$BQ(R)$] For $n \geq 1$, all finitely generated projective modules over $R[X_{1},...X_{n}]$ are extended from $R$.
\end{itemize}

It is expected that $BQ (R)$ holds whenever $R$ is a regular Noetherian ring. Note that all finitely generated projective $R[X]$-modules are free if $R$ is a regular local ring such that $BQ (R)$ holds. The Bass-Quillen conjecture is known to hold in many cases, e.g. if $R$ is a regular $k$-algebra essentially of finite type over a field $k$ (cp. \cite{Li}). Furthermore, it follows from the Quillen-Suslin theorem that all finitely generated projective $R[X]$-modules are free if $R$ is a regular local ring of dimension $\leq 1$. Moreover, M. P. Murthy proved in \cite{M} that all finitely generated projective $R[X]$-modules are free if $R$ is a regular local ring of dimension $2$ and later R. A. Rao proved in \cite{R} that the same statement holds if $R$ is a regular local ring of dimension $3$ with $6 \in R^{\times}$. Note that if $R$ is a regular local ring, the assumption on regularity implies that all finitely generated projective modules over $R[X]$ are stably free and hence the conjecture holds if and only if $GL_{r} (R[X])$ acts transitively on $\mathit{Um}_{r} (R[X])$ (or, equivalently, on $\mathit{Um}_{r}^{t} (R[X])$) for all $r \geq 3$. We may thus deduce the following statement from the previous results:

\begin{Prop}
Let $R$ be any regular local ring of dimension $4$ such that $6 \in R^{\times}$. Then all finitely generated projective $R[X]$-modules are free if and only if $Sp_{4} (R[X])$ acts transitively on $\mathit{Um}_{4}^{t} (R[X])$.
\end{Prop}

\begin{proof}
Since $R[X]$ is essentially of dimension $4$, we know that $E_{r} (R[X])$ acts transitively on $\mathit{Um}_{r} (R[X])$ for $r \geq 6$. Moreover, it was proven in \cite[Corollary 2.7]{R} that $E_{5} (R[X])$ acts transitively on $\mathit{Um}_{5} (R[X])$ as well.\\
If we let $P_{0} = R^{2}$, then there exists a canonical trivialization $\theta_{0}$ of $\det(R^{2})$ given by $1 \mapsto e_{1} \wedge e_{2}$, where $e_{1} = (1,0), e_{2} = (0,1) \in R^{2}$. Consequently, there is a generalized Vaserstein symbol $V_{\theta_{0}}: \mathit{Um}_{3} (R[X])/SL_{3} (R[X]) \rightarrow \tilde{V}_{SL} (R[X])$ associated to $\theta_{0}$. Although $\dim (R[X]) = 5$, the proofs of Theorems \ref{T3.5}, \ref{SLInjective} and Corollary \ref{C3.7} work for $R[X]$, because $E_{r} (R[X])$ acts transitively on $\mathit{Um}^{t}_{r} (R[X])$ for $r \geq 5$.\\
Now assume that all finitely generated projective $R[X]$-modules are free. Then $SL_{r} (R[X])$ acts transitively on $\mathit{Um}_{r}^{t} (R[X])$ for $r=3,4$. In particular, the orbit space $\mathit{Um}_{3} (R[X])/SL_{3} (R[X])$ is trivial. Then it follows directly from Corollary \ref{C3.7} that $Sp_{4} (R[X])$ acts transitively on $\mathit{Um}_{4}^{t} (R[X])$.\\
Conversely, assume only that $Sp_{4} (R[X])$ acts transitively on $\mathit{Um}_{4}^{t} (R[X])$. The proofs of \cite[Proposition 2.2 and Proposition 2.9]{R} show that the usual Vaserstein symbol $V_{-\theta_{0}}$ and hence also $V_{\theta_{0}}: \mathit{Um}_{3} (R[X])/SL_{3} (R[X]) \rightarrow \tilde{V}_{SL} (R[X])$ is a constant map. But the proof of Theorem \ref{SLInjective} then shows that it is also injective, because $Sp_{4} (R[X])$ acts transitively on $\mathit{Um}_{4}^{t} (R[X])$. Consequently, all finitely generated projective $R[X]$-modules are free.
\end{proof}

\subsection{Descriptions of the orbit spaces $Um (P_{0} \oplus R)/E (P_{0} \oplus R)$ and $Um (P_{0} \oplus R)/SL (P_{0} \oplus R)$}\label{3.3}

Let $R$ be a Noetherian commutative ring of dimension $\leq 4$ such that $SL (P_{5})$ acts transitively on the set $Um (P_{5})$. We now try to use the previous results in order to give descriptions of both the orbit spaces $Um (P_{0} \oplus R)/E (P_{0} \oplus R)$ and $Um (P_{0} \oplus R)/SL (P_{0} \oplus R)$.\\
For any map $F: M \rightarrow N$ between sets $M$ and $N$, one has $M = \cup_{x \in N} F^{-1} (x)$. Therefore we also have $Um (P_{0} \oplus R)/E (P_{0} \oplus R) = \cup_{\beta \in \tilde{V}_{SL} (R)} V^{-1} (\beta)$. Now let us fix an element $a \in Um (P_{0} \oplus R)$ together with a section $s$ of $a$ and give a description of the preimage $V^{-1} (V(a)) \subset Um (P_{0} \oplus R)/E (P_{0} \oplus R)$. We set $\chi = V (a,s)$. We have an obvious map

\begin{center}
$i_{a}: SL (P_{4}) \rightarrow V^{-1} (V(a)), \varphi \mapsto a \cdot \varphi$,
\end{center}

induced by the right action of $SL (P_{4})$ on $Um (P_{0} \oplus R)/E (P_{0} \oplus R)$. By Theorem \ref{T3.5}, this map is immediately surjective.\\
Now let $\varphi_{1}$ and $\varphi_{2}$ be two elements of $SL (P_{4})$ such that $\varphi_{1} {\varphi_{2}}^{-1} \in {Sp} (\chi)E (P_{4})$. Then obviously $i_{a} (\varphi_{1}) = i_{a} (\varphi_{2})$. Conversely, let $\varphi_{1}, \varphi_{2} \in SL (P_{4})$ such that $i_{a} (\varphi_{1}) = i_{a} (\varphi_{2})$. Then it follows from the proofs of \cite[Theorems 4.1 and 4.3]{Sy} that there is an element $\varphi \in E (P_{4})$ such that

\begin{center}
${\varphi_{1}}^{t}\chi \varphi_{1} = {\varphi}^{t} {\varphi_{2}}^{t} \chi \varphi_{2} \varphi$.
\end{center}

In particular, since $E (P_{4})$ is a normal subgroup of $SL (P_{4})$, it follows that $\varphi_{1} {\varphi_{2}}^{-1}$ lies in ${Sp} (\chi)E (P_{4})$. Thus, it follows that $i_{a}$ induces a bijection

\begin{center}
$i_{a}: {Sp} (\chi)E (P_{4})\backslash SL (P_{4}) \xrightarrow{\cong} V^{-1} (V(a))$
\end{center}

between the set of right cosets of ${Sp} (\chi)E (P_{4})$ in $SL (P_{4})$ and the preimage $V^{-1} (V(a))$. Altogether, we have just established the following description of $Um (P_{0} \oplus R)/E (P_{0} \oplus R)$:

\begin{Thm}
Let $\{\chi_{i}\}_{i \in I}$ be a set of non-degenerate alternating forms on $P_{4}$ such that $I \rightarrow \tilde{V}_{SL} (R), i \mapsto [P_{4}, \chi_{0} \perp \psi_{2}, \chi_{i}]$, is a bijection. Then there is a bijection $Um (P_{0} \oplus R)/E (P_{0} \oplus R) \cong \cup_{i \in I} {Sp} (\chi_{i})E (P_{4}) \backslash SL (P_{4})$.
\end{Thm}

\begin{Rem}
We remark that $SL (P_{4})/E (P_{4})$ is abelian if $R$ is a smooth affine algebra of dimension $4$ over an algebraically closed field $k$ such that $6 \in k^{\times}$ and $P_{0}$ is free: This follows from the fact that the map $SL_{4} (R)/E_{4} (R) \rightarrow SK_{1} (R)$ is injective in this situation (cp. \cite[Corollary 7.7]{FRS}). Hence the subgroup ${Sp} (\chi)E_{4} (R)$ of $SL_{4} (R)$ is normal and ${Sp} (\chi)E_{4} (R) \backslash SL_{4} (R) = SL_{4}/{Sp} (\chi)E_{4} (R)$.
\end{Rem}

Let us now describe the orbit space $Um (P_{0} \oplus R)/SL (P_{0} \oplus R)$. Analogously, we consider the surjective map $V: Um (P_{0} \oplus R)/SL (P_{0} \oplus R) \rightarrow \tilde{V}_{SL} (R)$ and describe the preimages $V^{-1} (V(a))$ for $a \in Um (P_{0} \oplus R)$. Henceforth we assume that $SL (P_{4})/E (P_{4})$ is an abelian group. By repeating the arguments above appropriately, we obtain a bijection

\begin{center}
$i_{a}: SL (P_{4})/{Sp} (\chi)SL (P_{3})E (P_{4}) \xrightarrow{\cong} V^{-1} (V(a))$.
\end{center}

\begin{Thm}
Let $\{\chi_{i}\}_{i \in I}$ be a set of non-degenerate alternating forms on $P_{4}$ such that $I \rightarrow \tilde{V}_{SL} (R), i \mapsto [P_{4}, \chi_{0} \perp \psi_{2}, \chi_{i}]$, is a bijection. Furthermore, assume that $SL (P_{4})/E (P_{4})$ is an abelian group. Then there is a bijection $Um (P_{0} \oplus R)/SL (P_{0} \oplus R) \cong \cup_{i \in I} SL (P_{4})/{Sp} (\chi_{i})SL (P_{3})E (P_{4})$.
\end{Thm}

\begin{Kor}
Let $R$ be a smooth affine algebra of dimension $\leq 4$ over an algebraically closed field $k$ of characteristic $\neq 2,3$. Furthermore, let $\{\chi_{i}\}_{i \in I}$ be a set of non-degenerate skew-symmetric forms on $R^{4}$ such that the map $I \rightarrow \tilde{V}_{SL} (R), i \mapsto [R^{4}, \chi_{0} \perp \psi_{2}, \chi_{i}]$, is a bijection. Then there is a bijection $\mathit{Um}_{3} (R)/SL_{3} (R) \cong \cup_{i \in I} SL_{4} (R)/{Sp} (\chi_{i})SL_{3} (R)E_{4} (R)$.
\end{Kor}

\subsection{Equality of linear and symplectic orbits}\label{Equality of linear and symplectic orbits}\label{3.4}

Let $R$ be a smooth affine algebra of even dimension $d$ over an algebraically closed field $k$ with $d! \in k^{\times}$. Motivated by the previous sections, we then study the orbits of unimodular rows of length $d$ under the right actions of $SL_{d} (R)$ and $Sp_{d} (R)$. We will use this to prove the equality $SL_{d} (R) e_{d} = Sp_{d} (R) e_{d}$. Since we have $SL_{d} (R) e_{d} = \mathit{Um}_{d} (R)$ in this case (cp. \cite[Theorem 7.5]{FRS}), this means that one has to prove that $Sp_{d} (R)$ acts transitively on the left on $\mathit{Um}_{d}^{t} (R)$.\\
As indicated, we will approach this problem in terms of the right actions of $SL_{d} (R)$ and $Sp_{d} (R)$ on $\mathit{Um}_{d} (R)$. For the remainder of this section, we let $\pi_{1,d}=(1,0,...,0)$ and $\pi_{d,d}=(0,...,0,1)$ be the standard unimodular rows of length $d$ and $e_{1,d} = \pi_{1,d}^{t}$ and $e_{d,d} = \pi_{d,d}^{t}$ the corresponding unimodular columns. As a first step, let us recall some basic facts about symplectic and elementary symplectic orbits. The following result by Gupta is a special case of \cite[Theorem 3.9]{G} and extends \cite[Theorem 5.5]{CR}:

\begin{Thm}
Let $R$ be a commutative ring. For any $n \in \mathbb{N}$ and unimodular row $v \in \mathit{Um}_{2n} (R)$, the equality $v E_{2n} (R) = v E{Sp}_{2n} (R)$ holds.
\end{Thm}

\begin{Kor}
Let $R$ be a commutative ring. If $v,v' \in \mathit{Um}_{2n} (R)$ for some $n \in \mathbb{N}$ and $v E_{2n} (R) = v' E_{2n} (R)$, then $v Sp_{2n} (R) = v' Sp_{2n} (R)$.
\end{Kor}

\begin{Thm}\label{T3.17}
Let $R$ be a smooth algebra of dimension $d \geq 4$ over an algebraically closed field $k$ with $d! \in k^{\times}$. Assume that $d$ is divisible by $4$. Then $Sp_{d} (R)$ acts transitively on $\mathit{Um}_{d} (R)$.
\end{Thm}

\begin{proof}
It follows from the proof of \cite[Theorem 7.5]{FRS} that any unimodular row of length $d$ can be transformed via elementary matrices to a row of the form $(a_{1},...,a_{d-1}, a_{d}^{{(d-1)!}^{2}})$. By the previous corollary, it thus suffices to show that any such row of length $d$ is the first row of a symplectic matrix.\\
So let $a = (a_{1},...,a_{d-1}, a_{d}^{(d-1)!})$ and let $b = (b_{1}, ..., b_{d-1},b_{d})$ be a unimodular row such that $a b^{t} = 1$. Furthermore, let $a' = (a_{1},...,a_{d-1}, a_{d}^{{(d-1)!}^{2}})$. It follows from \cite[Proposition 2.2, Corollary 2.5]{S4} that there exists a matrix $\beta (a,b) \in SL_{d} (R)$ whose first row is $a'$ such that $[\beta (a,b)] = [\alpha_{d} (a,b)]$ in $SK_{1} (R)$.\\
Now let us first assume that the class of $\alpha_{d} (a,b)$ in $K_{1} (R)$ lies in the image of the forgetful map $K_{1}{Sp} (R) \xrightarrow{f} K_{1} (R)$. As a matter of fact, it is well-known that ${Sp} (R) = {ESp}(R){Sp}_{d}(R)$ (cp. \cite[Theorem 7.3(b)]{SV}). Hence the class of $\alpha_{d} (a,b)$ lies in the image of the composite $Sp_{d} (R) \rightarrow K_{1}{Sp}(R) \xrightarrow{f} K_{1} (R)$. In other words, there is a matrix $\varphi \in Sp_{d} (R)$ with $[\varphi] = [\alpha_{d} (a,b)] = [\beta (a,b)]$ in $K_{1} (R)$. As the homomorphism $SL_{d} (R)/E_{d} (R) \rightarrow SK_{1} (R)$ is injective (cp. \cite[Corollary 7.7]{FRS}), it follows that $\beta (a,b) {\varphi}^{-1} \in E_{d} (R)$. Since furthermore the equality $\pi_{1,d} E_{d} (R) = \pi_{1,d} {ESp}_{d} (R)$ holds, there is $\psi \in {ESp}_{d} (R)$ such that $\pi_{1,d}\beta (a,b) {\varphi}^{-1} = \pi_{1,d} \psi$. In particular, $a' = \pi_{1,d} \beta (a,b) = \pi_{1,d} \psi \varphi$ lies in the orbit of $\pi_{1,d}$ under the action of $Sp_{d} (R)$.\\
Thus, it suffices to show that the class of $\alpha_{d} (a,b)$ in $K_{1} (R)$ indeed lies in the image of $K_{1}{Sp}(R) \xrightarrow{f} K_{1} (R)$. For this purpose, recall from Section \ref{2.4} that we have canonical identifications $Um_{d} (R) \cong Hom_{Sm_{k}} (X, \mathbb{A}^{d}\setminus 0)$ and $Um_{d} (R)/E_{d} (R) \cong [X, \mathbb{A}^{d}\setminus 0]_{\mathbb{A}^{1}}$. Hence a unimodular row of length $d$ over $R$ corresponds to a morphism $X = Spec (R) \rightarrow \mathbb{A}^{d}\setminus 0$ and there is a canonical pointed $\mathbb{A}^{1}$-weak equivalence $p_{2d-1} : Q_{2d-1} \rightarrow \mathbb{A}^{d}\setminus0$. As a matter of fact, a morphism $X \rightarrow Q_{2d-1}$ corresponds to a unimodular row of length $d$ with the choice of an explicit section. Furthermore, there is an $\mathbb{A}^{1}$-fiber sequence $Sp \rightarrow GL \rightarrow GL/Sp$, which induces the Karoubi periodicity sequence by taking the sets of morphisms in $H (k)$. Moreover, there is a pointed morphism $\alpha_{d}: Q_{2d-1} \rightarrow SL \hookrightarrow GL$ induced by $\alpha_{d} (x,y)$.\\
Let $a'' = (a_{1},..., a_{d-1}, a_{d}) \in \mathit{Um}_{d} (R)$. We now interpret this unimodular row as a morphism $a'': X \rightarrow \mathbb{A}^{d}\setminus 0$ of spaces. If we let $\Psi^{(d-1)!}: \mathbb{A}^{d}\setminus 0 \rightarrow \mathbb{A}^{d}\setminus 0$ be the morphism induced by $(x_{1},...,x_{d-1}, x_{d}) \mapsto (x_{1},...,x_{d-1},x_{d}^{(d-1)!})$, then we obviously have $a = \Psi^{(d-1)!} a'': X \rightarrow \mathbb{A}^{d}\setminus 0$. It thus suffices to prove the existence of a morphism $\mathbb{A}^{d}\setminus 0 \rightarrow Sp$ in $H (k)$ that makes the diagram

\begin{center}
$\begin{xy}
  \xymatrix{
      \mathbb{A}^{d}\setminus 0 \ar[r]^{p_{2d-1}^{-1} \Psi^{(d-1)!}} \ar@{.>}[d]    & Q_{2d-1} \ar[d]^{\alpha_{d}}  \\
      Sp \ar[r]             &   GL \ar[r] &  GL/Sp
  }
\end{xy}$
\end{center}

commutative. For this purpose, we first of all note that the motivic Brouwer degree of $\Psi^{(d-1)!} \in [\mathbb{A}^{d}\setminus 0, \mathbb{A}^{d} \setminus 0]_{\mathbb{A}^{1},\bullet} = GW (k)$ is ${(d-1)!}_{\epsilon}$. Since $k$ is algebraically closed, we know that $(d-1)!_{\epsilon} = (d-1)! \in GW(k)$. Hence it follows that $\alpha_{d} p_{2d-1}^{-1}\Psi^{(d-1)!}$ equals $(d-1)! \cdot \alpha_{d}p_{2d-1}^{-1} \in [\mathbb{A}^{d}\setminus 0, GL]_{\mathbb{A}^{1},\bullet}$, where the group structure is understood with respect to the structure of $\mathbb{A}^{d}\setminus 0$ as an $h$-cogroup in $H_{\bullet} (k)$. The usual Eckmann-Hilton argument then implies that also $\alpha_{d} p_{2d-1}^{-1}\Psi^{(d-1)!} = (d-1)! \cdot \alpha_{d}p_{2d-1}^{-1}$ in $[\mathbb{A}^{d}\setminus 0, GL]_{\mathbb{A}^{1},\bullet}$, where the group structure is understood with respect to the structure of an $h$-group of $GL \simeq_{\mathbb{A}^{1}} \mathcal{R}\Omega_{s}BGL$ in $H_{\bullet} (k)$. This is the group structure corresponding to the identification $[\mathbb{A}^{d}\setminus 0, GL]_{\mathbb{A}^{1}} \cong K_1 (\mathbb{A}^{d}\setminus 0)$. As $[\mathbb{A}^{d}\setminus 0, GL/Sp]_{\mathbb{A}^{1}} \cong W_E (S_{2d-1}) \cong \mathbb{Z}/2\mathbb{Z}$ and $(d-1)!$ is even, it follows that

\begin{center}
${(\mathbb{A}^{d}\setminus 0)} \xrightarrow{p_{2d-1}^{-1}\Psi^{(d-1)!}} Q_{2d-1} \xrightarrow{\alpha_{d}} GL \rightarrow GL/Sp$
\end{center}

is trivial and hence the factorization exists, as desired.
\end{proof}

As a consequence, we can prove a corresponding statement for the left action of $Sp_{d} (R)$ on $\mathit{Um}_{d}^{t} (R)$:

\begin{Kor}\label{C3.18}
Let $R$ be a smooth affine algebra of dimension $d \geq 4$ over an algebraically closed field $k$ with $d! \in k^{\times}$. Assume that $d$ is divisible by $4$. Then $Sp_{d} (R)$ acts transitively on $\mathit{Um}_{d}^{t} (R)$; in particular, $Sp_{d} (R) e_{d} = SL_{d} (R) e_{d}$.
\end{Kor}

\begin{proof}
First of all, let

\begin{center}
$\varphi_{2} =
\begin{pmatrix}
-1 & 0 \\
0 & 1
\end{pmatrix}
\in GL_{2} (R)$.
\end{center}

We can then inductively define $\varphi_{2n+2} = \varphi_{2n} \perp \varphi_{2} \in GL_{2n+2} (R)$ for all $n \in \mathbb{N}$. Furthermore, we have $\varphi_{d}^{t} \psi_{d} \varphi_{d}= \psi_{d}^{t}$, $\varphi_{d}^{t} = \varphi_{d}$ and $\varphi_{d}^{-1} = \varphi_{d}$.\\
Now let $v \in \mathit{Um}_{d} (R)$ and $v^{t}$ the corresponding unimodular column. By Theorem \ref{T3.17}, there is $\varphi \in Sp_{d} (R)$ with $\pi_{d,d} \varphi = v \varphi_{d}$. It follows that $\varphi_{d} \varphi^{t} \varphi_{d} \in Sp_{d} (R)$. Finally, one has $\varphi_{d} \varphi^{t} \varphi_{d} e_{d,d} = \varphi_{d} \varphi_{d} v^{t} = v^{t}$, which proves the corollary.
\end{proof}

\begin{Thm}\label{T3.19}
Let $R$ be a $4$-dimensional smooth affine algebra over an algebraically closed field $k$ with $6 \in k^{\times}$. Then stably free $R$-modules of rank $2$ are free if and only if $\tilde{V}_{SL} (R) = 0$.
\end{Thm}

\begin{proof}
By \cite[Theorem 7.5]{FRS}, stably free $R$-modules of rank $2$ are free if and only if $Um_{3} (R)/GL_{3} (R)$ is trivial; clearly, this holds if and only if the orbit space $Um_{3} (R)/SL_{3} (R)$ is trivial. Hence the theorem follows immediately from Corollary \ref{C3.7} and Corollary \ref{C3.18}.
\end{proof}

\begin{Kor}\label{C3.20}
Let $R$ be a $4$-dimensional smooth affine algebra over an algebraically closed field $k$ with $6 \in k^{\times}$ and let $X = Spec(R)$. Then $\mathit{Um}_{3} (R)/SL_{3} (R)$ is trivial if $CH^{3} (X)$ and $H^{2} (X, \textbf{\textsc{K}}_{3}^{MW})$ are $2$-divisible. Furthermore, the orbit space $\mathit{Um}_{3} (R)/SL_{3} (R)$ is trivial if $CH^{3} (X) = CH^{4} (X) = 0$ and $H^{2} (X, \textbf{\textsc{I}}^{3})$ is $2$-divisible.
\end{Kor}

\begin{proof}
By Theorem \ref{T3.19}, we have to show that $W_{SL} (R) \cong \tilde{V}_{SL} (R) = 0$ if $CH^{3} (X)$ and $H^{2} (X, \textbf{K}_{3}^{MW})$ are $2$-divisible or if $CH^{3} (X) = CH^{4} (X) = 0$ and $H^{2} (X, \textbf{I}^{3})$ is $2$-divisible. But since the Vaserstein symbol surjects onto $W_{SL} (R)$ and $k$ is algebraically closed, it follows from \cite[Lemma 7.4]{FRS} and the Swan-Towber theorem \cite[Theorem 2.1]{SwT} that $W_{SL} (R)$ is $2$-torsion. Hence it suffices to show that $W_{E} (R)$ or $W_{SL} (R)$ is $2$-divisible. So the first statement follows from Propositions \ref{P2.6} and \ref{P2.8}. The second statement follows directly from Proposition \ref{P2.8}.
\end{proof}

\begin{Kor}\label{C3.21}
Let $R$ be a $4$-dimensional smooth affine algebra over an algebraically closed field $k$ with $6 \in k^{\times}$ and let $X = Spec(R)$. Moreover, assume that $CH^{i} (X) = 0$ for $i=1,2,3,4$ and that $H^{2} (X, \textbf{\textsc{I}}^{3}) = 0$. Then all finitely generated projective $R$-modules are componentwise free.
\end{Kor}

\begin{proof}
We denote by $[X,\mathbb{Z}]$ the group of continuous maps $X \rightarrow \mathbb{Z}$. We may assume that $X = Spec (R)$ is connected; in particular $[X,\mathbb{Z}] \cong \mathbb{Z}$. The fact that $CH^{i} (X) = 0$ for $i=1,2,3,4$ implies that $F^{i} K_{0} (R) = 0$ for $i=1,2,3,4$. Hence $rank: K_{0} (R) \xrightarrow{\cong} [X,\mathbb{Z}]$ is an isomorphism and all finitely generated projective $R$-modules are stably free. Since stably free $R$-modules of rank $\geq 3$ are free (cp. \cite{S1} and \cite{FRS}), it suffices to prove that stably free modules of rank $2$ are free. But this follows from Corollary \ref{C3.20}.
\end{proof}

\begin{Rem}\label{R3.22}
The generalized Serre conjecture on algebraic vector bundles asserts that algebraic vector bundles on a topologically contractible smooth affine complex variety $X$ are trivial. If $dim(X) \leq 2$, this conjecture is known to hold; the conjecture remains open in higher dimension. Moreover, if $dim (X) = 3$, all algebraic vector bundles on $X$ are trivial if and only if $CH^{2} (X) = CH^{3} (X) = 0$. If $dim (X) = 4$, it follows from Corollary \ref{C3.21} that all algebraic vector bundles on $X$ are trivial if $CH^{2} (X) = CH^{3} (X) = CH^{4} (X) = 0$ and $H^{2} (X, \mathbf{I}^{3}) = 0$ (one always has $CH^{1} (X) = 0$ by \cite[Theorem 5.2.7]{AO}). We refer the reader to \cite[Section 5.2]{AO} for a nice survey of the results on this conjecture.\\
Another open conjecture asserts that topologically contractible smooth affine complex varieties are stably $\mathbb{A}^{1}$-contractible (cp. \cite[Conjecture 5.3.11]{AO}). The conditions in Corollary \ref{C3.21} are satisfied if $X$ is a stably $\mathbb{A}^{1}$-contractible smooth affine variety over an algebraically closed field: Since the restriction of $\mathbf{I}^{5}$ to the small Nisnevich site of $X$ is trivial (cp. \cite[Proposition 5.1]{AF2}), the long exact cohomology sequence associated to $0 \rightarrow \mathbf{I}^{4}/\mathbf{I}^{5} \rightarrow \mathbf{I}^{3}/\mathbf{I}^{5} \rightarrow \mathbf{I}^{3}/\mathbf{I}^{4} \rightarrow 0$ implies that $H^{2} (X, \mathbf{I}^{3}) = 0$ as soon as $H^{2} (X, \mathbf{K}^{M}_{4}/2) = H^{2} (X, \mathbf{K}^{M}_{3}/2) = 0$ (because of Voevodsky's resolution of the Milnor conjectures). But since $X$ is assumed to be stably $\mathbb{A}^{1}$-contractible, all the cohomology groups $H^{i} (X,\mathbf{K}^{M}_{j}/2)$ and $H^{i} (X,\mathbf{K}^{M}_{j})$ (and hence also Chow groups) vanish. Altogether, it follows from Corollary \ref{C3.21} that \cite[Conjecture 5.3.11]{AO} would imply the generalized Serre conjecture on algebraic vector bundles in dimension $4$.\\
Finally, let us mention that explicit examples of stably $\mathbb{A}^1$-contractible smooth affine varieties of dimension $4$ over algebraically closed fields of characteristic $0$ have been constructed in \cite[Theorem 1.19]{DPO}. It also follows from Corollary \ref{C3.21} that a smooth affine complex variety of dimension 4 whose motive is isomorphic to that of a point has only trivial vector bundles, because any such variety is stably $\mathbb{A}^1$-contractible. A construction of such varieties has been given in \cite[Example 6]{As}.
\end{Rem}

\end{document}